\documentclass[11pt, a4paper]{amsart}

\usepackage[left=3cm, top=3cm,bottom=3cm,right=3cm]{geometry}

\usepackage[utf8]{inputenc}
\usepackage{lmodern}
\usepackage{amssymb,amsfonts,amsthm,amsmath,bbm,mathrsfs,aliascnt,mathtools}
\usepackage[english]{babel}
\usepackage[colorlinks]{hyperref}

\usepackage{ulem}
\usepackage{nicefrac}

\theoremstyle{plain}
\newtheorem{thm}{Theorem}[section]
\newtheorem{cor}[thm]{Corollary}
\newtheorem{lem}[thm]{Lemma}
\newtheorem{prop}[thm]{Proposition}

\theoremstyle{definition}

\theoremstyle{remark}

\DeclareMathOperator{\Hom}{Homeo}

\DeclareMathOperator{\supp}{supp}
\DeclareMathOperator{\card}{Card}

\newcommand{\uA}{\underline{A}}

\newcommand{\caseA}{${\mathfrak{a}}$}
\newcommand{\caseB}{$\mathfrak{b}$}

\newcommand{\I}{\mathcal I}
\newcommand{\E}{\mathbb E}

\newcommand{\R}{\mathbb R}
\newcommand{\Z}{\mathbb Z}
\newcommand{\N}{\mathbb N}

\renewcommand{\P}{\mathbb P}

\newcommand{\g}{\mathbf{g}}
\newcommand{\ug}{\underline{g}}
\newcommand{\HH}{\mathcal H}
\newcommand{\II}{\mathcal I}

\newcommand{\EE}{\mathcal E}
\newcommand{\FF}{\mathcal F}

\newcommand{\ux}{\underline{x}}

\newcommand{\Homr}{\Hom^+(\mathbb R)}

\newcommand{\ps}[1]{\hspace{-1.2pt}{}^{{\scalebox{.62}{ $#1$}}}\hspace{-2.2pt}}
\newcommand{\Aplus}{\ps{{+}}A}
\newcommand{\Aminus}{\ps{{-}}A}
\newcommand{\Yplus}{\ps{{+}}\, Y}
\newcommand{\Yminus}{\ps{{-}}\,Y}

\newcommand{\Cont}{$\mathfrak{C}$}
\newcommand{\Rec}{$\mathfrak{R}$}
\newcommand{\UB}{$\mathfrak{U}$}

\DeclareMathOperator*\uplim{\overline{lim}}

\def\d{{\rm d}}

\begin{document}

\title{On uniqueness of invariant measures for random walks on ${\textup {HOMEO}}^+(\mathbb R)$.}

\author{Sara Brofferio}
\address{Sara Brofferio, Laboratoire de Math\'ematiques, Universit\'e Paris-Sacley, Campus d’Orsay, France.}
\email{sara.brofferio@gmail.com}

\author{Dariusz Buraczewski}
\address{Dariusz Buraczewski, Mathematical Institute University of Wroclaw, Pl. Grunwaldzki 2/4
50-384 Wroclaw, Poland}\email{dariusz.buraczewski@uwr.edu.pl}

\author{Tomasz Szarek}
\address{Tomasz Szarek, Institute of Mathematics Polish Academy of Sciences, Abrahama 18, 81-967 Sopot, Poland}
\email{tszarek@impan.pl}

\maketitle

\begin{abstract}
We consider random walks  on the group  of
orientation-preserving homeomorphisms of the real line $\R$.
In particular, the fundamental question of uniqueness
of an invariant measure of the generated process is raised.
This problem was already studied  by Choquet
and Deny (1960) in the context of random walks generated by translations of the line.
Nowadays the answer is quite well understood in general settings
of strongly contractive systems. Here we focus on broader class of systems
satisfying the conditions: recurrence, contraction and unbounded action. We prove that under these 
conditions the random process possesses a unique invariant Radon measure on $\R$.
Our work can be viewed as a subsequent paper of Babillot  et al. (1997)
and Deroin et al. (2013).
\end{abstract}





\section{Introduction}

Let $\Hom^+(\mathbb R)$ denotes the group of orientation-preserving homeomorphisms  of the real line $\mathbb R$.
We shall consider the (left) random walk on  $\Hom^+(\mathbb R)$, that is the sequence of random homeomorphims
$$
\ell_n:=\g_n\cdots \g_1
$$
obtained by iterated composition products of a sequence $(\g_n)_{n\in\N}$  of independent identically distributed  $\Hom^+(\mathbb R)$--valued random variables. We denote by $\mu$ the common distribution of the $\g_n$'s. We shall always assume that $\mu$  is a \textit{discrete} probability measure on $\Homr$.
This sequence of random transformations induces a \textit{Stochastic Dynamical System} (or {\it Iterated Random Function Systems}) on the real line, that is the Markov chain  $(X_n^x)_{n\in\N}$ defined recursively
for any starting value   $X_0^x=x\in\mathbb R$ by the formula:
\begin{equation*}
X_n^x:= \g_n(X_{n-1}^x)=\ell_n(x)\quad \mbox{for $n\ge 1$}.
\end{equation*}
The associated Markov kernel   is of the form
\begin{equation*}
Pf(x):=\sum_{g\in\Gamma} f\left(g (x)\right)\mu(g)\quad \mbox{for any bounded Borel-measurable function $f$ on $\R$.}
\end{equation*}
Here $\Gamma$ denotes the discrete support of $\mu$, that is $\Gamma:=\{g\in\Homr|\ \mu(g)>0\}$.

 We are interested in the case when the Markov chain $(X_n^x)_{n\in\mathbb N}$ does not escape to infinity. Namely, we  always suppose that the following hypothesis is satisfied:
\begin{itemize}
	\item[($\mathfrak{R}$)] The Markov chain  is (uniformly topologically) {\it recurrent}, that is   there exists a~compact interval $\I\subset\mathbb R$ such that for every $x\in\mathbb R$ the sequence $(X_n^x)_{n\in\mathbb N}$ visits $\I$ infinitely often a.s.
\end{itemize}
Condition ($\mathfrak{R}$) entails immediately that there exists an {\it invariant Radon measure} $\nu$ for the System generated by $\mu$, i.e.~a  measure,  finite on compact sets,  satisfying
$$
\int_\R f(x)\d\nu( x) = \int_{\R}Pf(x)\d\nu( x)
$$
for any $f\in C_C(\R)$, the space of continuous functions with compact support.
This measure can be either finite or infinite.

The fundamental question of this paper is to decide whether an invariant measure is unique up to a multiplicative constant. This problem has been widely studied for different kind of systems: the now classical Choquet-Deny theorem (\cite{Choquet:Deny}) can been seen as one of~the first results in this direction. It says that the Lebesgue measure on $\R$ is the unique Radon invariant measure for systems generated by translations that are recurrent  and do not have discrete orbits. Among other interesting results we would like to mention fundamental works on strongly contractive systems initiated  by H. Furstenberg \cite{Furstenberg},
see also \cite{Diaconis:Freedman,Letac,Peigne:Woess2}. In these works, under various contracting assumptions, it was proved that there exists a unique invariant probability measure.

A weaker contraction property (called local stability) has been proposed to deal with systems that have infinite Radon invariant measures. At first this property was used by Babillot, Bougerol and Elie \cite{Babillot:Bougerol:Elie} in the case of systems generated by centred random affinities. Next it was studied in much more general setting by M. Benda \cite{Benda}, M. Peigné and W. Woess \cite{Peigne:Woess1}, and B. Deroin, V. Kleptsyn, A. Navas, K. Parwani  \cite{Deroin:Kleptsyn:Navas:Parwani:2013}.
The latter paper contains a detailed study of the uniqueness of an invariant measure for random walks on $\Homr$ under the hypothesis of the measure $\mu$ being symmetric.
Using a~conjugation of the reals to some open interval, say $(0, 1)$,  we obtain some results for random walks on the group of orientation preserving homeomorphisms of the interval $(0, 1)$.
 At~first such walks were considered by L. Alsed\'a and M. Misiurewicz who studied some function systems consisting of piecewise linear homeomorphisms and proved the existence of a unique probability measure (see \cite{alseda_misiurewicz}). More general function systems were investigated by M. Gharaei and A. I. Homburg in \cite{Gharaei_Homburg}.
Recently D. Malicet obtained unique ergodicity as a consequence of the contraction principle for time homogeneous random walks on the topological group of homeomorphisms defined on the circle and interval  (see \cite{malicet}). His proof, in turn, is based upon an invariance principle of A. Avila and M. Viana (see \cite{Avila_Viana}). A  simple proof of unique ergodicity on the open interval $(0, 1)$ for a wide class of iterated function systems is given in \cite{Czudek_Szarek}.

The main goal of this paper is to show that the uniqueness of an invariant measure can be obtained  assuming, besides recurrence, the following two conditions that only involve the action of $\Gamma$ (the  support of $\mu$) on $\R$:
\begin{itemize}
\item[($\mathfrak{C}$)] \textbf{Contraction (or proximality) of the action.} There exists an interval $\I\subset\mathbb R$ such that for any compact set $K\subset\mathbb R$ there is some $g$ belonging to the semigroup generated by  $\Gamma$ such that $g(K)\subset \I$.
	\item[($\mathfrak{U}$)] \textbf{Unboundedness of the action.}  For every $x\in\mathbb R$ we have $g_1(x)<x< g_2(x)$ for some $g_1, g_2\in \Gamma$.
\end{itemize}	
The first of this conditions says that it is possible to shrink any bounded set at finite distance.  We will see that the second condition is equivalent to  the question of whether one can reach $+\infty$ and $-\infty$ from any starting point  $x$.

From now on, uniqueness will mean the existence of a unique, up to a scalar factor, invariant Radon measure.
The main purpose of the paper is to prove that under the above conditions the invariant measure is unique:

\begin{thm}\label{thm:main} Assume that a Stochastic Dynamical System, generated by a discrete distribution  $\mu$ on $\Homr$, satisfies assumptions ($\mathfrak{R}$), ($\mathfrak{C}$) and  ($\mathfrak{U}$). Then the System admits a unique invariant Radon measure $\nu$.
\end{thm}

The study of invariant measures is strictly related to the issue of  closed $\Gamma$-invariant sets, that is, closed sets  $M\subseteq\R$ such that $g M\subseteq M$ for all $g\in \Gamma$. In fact, for any invariant measure $\nu$ its support $\supp\nu$ is a closed $\Gamma$-invariant subset of $\R$. One of the crucial questions that the present paper  explores is  whether a closed $\Gamma$-invariant set can be contained in the support of  different invariant ergodic measures. Recall that an invariant measure $\nu$ is {\it ergodic} if for any $A\subseteq \R$ such that $\nu_A$, the restriction of $\nu$ to $A$, is invariant,  we obtain that either $\nu_A=\nu$ or $\nu_A \equiv 0$.
The following theorem gives a quite complete answer to this question under the recurrence and unboundedness hypotheses only.
	\begin{thm}\label{thm:main2}
	Assume that a Stochastic Dynamical System, generated by a discrete distribution  $\mu$ on $\Homr$, satisfies assumptions ($\mathfrak{R}$) and ($\mathfrak{U}$).
		\begin{enumerate}\item Let $\nu_1$ and $\nu_2$ be two ergodic invariant Radon measures such that $\supp\nu_1 \subseteq\supp\nu_2$ and $\supp\nu_1$ is not discrete. Then $\nu_1=C\nu_2$ for some constant $C>0$.
			\item The support of every ergodic invariant Radon measure $\nu$ is either minimal among the closed $\Gamma$-invariant sets or it contains a $\Gamma$-invariant discrete set.
			\item For any minimal  closed $\Gamma$-invariant set $M$ there exists a unique ergodic invariant Radon measure $\nu$ such $M=\supp\nu$.
		\end{enumerate}
	\end{thm}
	
The proof of Theorem \ref{thm:main2} will be given in Section \ref{sec: Minimalty}. In Section \ref{sec: uniqueness} we will show that Theorem \ref{thm:main} is a consequence of this result together with the contraction hypothesis and the ergodic decomposition of invariant measures.  We would like to point out that the results of these two theorems are quite optimal and that conditions (\Rec),  (\Cont) and (\UB) are  all needed to ensure uniqueness. In Section \ref{sect: exp} we shall  provide a number of examples and discus our hypothesis.

In this paper we would also like to show how the general theorem -- Theorem \ref{thm:main}, can be applied to several specific but interesting situations. For instance we will prove  that an immediate consequence is the uniqueness of an invariant measure for recurrent affine recursions:
\begin{cor}\label{cor: aff}
Let $\mu$ be a discrete measure on $\Gamma\subset \Hom^+(\mathbb R)$. Assume that  every $g\in\Gamma$ is of the form $g(x)=A(g)x+B(g)$ for $x\in\mathbb R$. Moreover, assume that there exists $g_0\in \Gamma$ such that $A(g_0)<1$. Then, if conditions  (\Rec) and (\UB) hold, the corresponding Stochastic Dynamical System admits a unique invariant measure $\nu$.
\end{cor}
This result is well known but we give here a new proof of it. In particular, it is not based on the Lipschitz property of affine transformations. The proof is valid both in the contractive case (when there exists a stationary probability \cite{Furstenberg}) and in the centred case (when the invariant measure has infinite mass \cite{Babillot:Bougerol:Elie}).

Recurrence (\Rec) and contraction condition (\Cont)  can be easily verified when homeomorphisms  are repulsive at $\pm \infty$. In Lemma \ref{lem:condAL} we will present some general criteria for systems that are asymptotically linear, such as affine recursions.  As a consequence, using  a conjugation, one can obtain the following results for $C^2$-diffeomorphism of the interval.
\begin{cor}\label{cor:C2diff}
	Let  $\mu$ be a finitely  supported measure on the group of increasing diffeomorphism in $C^2([0,1])$. Assume that
	\begin{itemize}
		\item[(\Rec')] $\displaystyle
		\sum_{h\in\supp \mu}\mu(h)\ln h'(0)\geq 0\quad\text{and}\quad \sum_{h\in\supp \mu}\mu(h)\ln h'(1)\geq 0$;
		\item[(\Cont')] there exists $h\in\supp \mu$ such that  $h'(0)>1$ and $h'(1)>1$;
		\item[(\UB')]  for every $x\in (0, 1)$ there exist $h_1,h_2\in\supp \mu$   such that $h_1(x)<x<h_2(x)$.
	\end{itemize}
	Then there exists a unique invariant Radon measure on $(0,1)$.
\end{cor}

In Section \ref{sec:ergodic}, as an appendix, we shall discuss some seminal results on ergodic invariant measures  for Markov--Feller processes on locally compact metric spaces.
In particular, we will give an explicit proof of the ergodic decomposition of a general invariant Radon measure as an integral over all ergodic Radon measures.

\section{Basic notions and preliminary results}\label{sec:basic}

In this section we give the fundamental notions and basic facts about invariant Radon measures that will play an important role in the sequel.

\subsection{Random walks on $\Homr$ and associated dynamical systems}

We denote by $\Homr$ the set of orientation preserving homeomorphisms  of the real line.
We consider the left random walk on $\Hom^+(\mathbb R)$, i.e. the Markov Chain
$$
\ell_n:=\g_n\cdots \g_1
$$
obtained by composition product of a sequence $(\g_n)_{n\in\N}$, which is a sequence of i.i.d. $\Hom^+(\mathbb R)$--valued random variables whose distribution is a discrete measure $\mu$.
Let $$\Gamma:=\{g\in\Homr\,\,|\,\, \mu(g)>0 \}\subset \Hom^+(\mathbb R)$$ be the discrete support of $\mu$.
The space of trajectories of the random walk is then the infinite product space $\Gamma^{\mathbb N}$. This space will be equiped with the product measure: $(\Gamma^{\mathbb N}, \mu^{\otimes \mathbb N})$. The associated probability law will be denoted by~$\P$. We denote by $$\Gamma^*:=\{g=g_1\cdots g_n\in \Homr\,\,|\,\,\mbox{ for some } g_i\in \Gamma \}$$ the semigroup generated by $\Gamma$. Observe that $\Gamma^*$ is countable and may be equipped with the discrete topology.

We denote by $\mathcal B(\R)$ the collection of all Borel subsets of $\R$, by $B(\R)$
the family of all Borel--measurable bounded (real valued) functions with the supremum norm $\|\cdot\|_{\infty}$ and by $C(\R)$ the subspace of $B(\R)$ consisting of all continuous functions. The subfamily of $C(\R)$ consisting of all continuous functions with compact support is denoted by $C_C(\R)$.

Since the semigroup $\Gamma^*$ acts on $\R$, we can  introduce the Stochastic Dynamical System on the real line $(X_n^x)_{n\in\N}$  corresponding to the left random walk on $\Hom^+(\mathbb R)$, that is, for any $x\in\mathbb R$ we define the Markov chain
\begin{equation*}
X_n^x:= \g_n(X_{n-1}^x)=\ell_n(x)\quad \mbox{for $n\ge 1$}
\end{equation*}
and $X_0^x=x$.
The transition probability for this Markov chain is given by the formula:
$$
P(x, A)=\sum_{g\in \Gamma} {\bf 1}_A(g(x))\mu(g)\qquad \mbox{ for $x\in\R$ and $A\in {\mathcal B}(\R)$.}
$$
It induces a positive contraction $P$ on $B(\R)$ defined by
\begin{equation}\label{e2_20.02.13}
Ph(x):=\sum_{g\in\Gamma}h(g (x))\mu(g)\qquad \mbox{for $h\in B(\R)$}.
\end{equation}

For any Radon measure $\nu$ on $\R$, let $ \P_\nu$ be the measure defined on the trajectories of  the Markov chain $(X_n)_{n\in\mathbb N}$  where $X_0$ is distributed according to  $\nu$.
More precisely, $ \P_\nu$ is a measure on the space  $ \R^\N$  (endowed with the product $\sigma$-algebra) such that for any finite collection of compact intervals $I_i$, $i=0,\ldots, n$, the measure of the cylinder $[I]=I_0\times\cdots \times I_n\times \R\times \cdots$ is defined by:
\begin{align*}
\P_\nu([I])&= \P_\nu(X_0\in I_0,\ldots, X_n\in I_n )\\&:=\sum_{g_1,\ldots, g_n\in\Gamma} \mu(g_1)\cdots \mu(g_n)\int_\R {\bf 1}_{I_0}(x) {\bf 1}_{I_1}(g_1(x))\cdots {\bf 1}_{I_n}(g_n\cdots g_1(x))\d\nu(x).
\end{align*}
Observe that if $\nu$ is a Radon measure of infinite mass then $ \P_\nu$ is  not a probability measure but it is finite on the cylinders whose bases $I_i$ are compact intervals.

\subsection{Invariant measure and recurrence }
An {\it invariant Radon measure} for the system induced by a measure $\mu$ is a  Borel measure $\nu$ that is finite on compact sets and satisfies
$$
\nu(f) = \int_{\R}\E f(X_1^x)\nu(\d x) = \sum_{g\in \Gamma} \int_{\R} f(g(x))\nu(\d x)\mu(g)=\nu(Pf)
$$ for any $f\in C_C(\R)$. In short, we shall say the $\nu$ is {\it invariant} for $\mu$, or that $\nu$ is a $\mu$--invariant measure.
It is easy to check that if $\nu$ is invariant for $\mu$, then the measure $\P_\nu$ is invariant for the shift $\tau$ of $\R^\N$.

It is well known that recurrence hypothesis  (\Rec) entails immediately the existence of a $\mu$-invariant Radon measure.
 Indeed, it is easy to see that the operator $P$ is {\it topologically conservative}, that is, there exists a bounded set $K\subset\R$ such that
$$
\sum_{k=0}^{\infty} P^k{\bf 1}_K(x)=\infty \quad  \mbox{for every $x\in\R$}.
$$
Actually,  ($\mathfrak{R}$) implies that this condition holds with
$K=\I$. Then Lin's result \cite[Theorem 5.1]{Lin70} ensures the existence of a $\mu$--invariant Radon measure $\nu$.
\medskip

\subsection{Support of an invariant measure and closed $\Gamma$-invariant sets}\label{sec:Gamma-invariant-set}
The analysis of $\mu$-invariant measures is strictly related to the study of \textit{closed $\Gamma$- invariant sets} of $\R$, that is, closed sets $M\subseteq\R$ such that $gM\subseteq M$ for all $g\in \Gamma$.
In fact, for any $\mu$-invariant measure $\nu$  its support
$$
\supp\nu:=\{x\in\R: \nu(V_x)>0\,\,\mbox{ for every open neighbourhood $V_x$ of } x\}
$$
is a  closed   $\Gamma$-invariant set of $\R$. To check $\Gamma$-invariance, take $x\in \supp\nu$, $g_0\in \Gamma$ and $V$ an open neighbourhood of $g_0(x)$. Then
$$\nu(V)=\sum_{g\in\Gamma}\mu(g)\nu(g^{-1}V)\geq \mu(g_0)\nu(g_0^{-1}V)>0,$$
since $g_0^{-1}V$ is an open-neighbourhood of $x$.

If (\Rec) holds then, thanks to Lin's Theorem, any closed $\Gamma$-invariant set contains the support of at least one $\mu$-invariant Radon measure $\nu$. In particular, if there exist two disjoint closed $\Gamma$-invariant sets, there are at least two different invariant measures.

To decide whether a $\Gamma$-invariant set can be (or contains) the support of  different invariant measures it is indispensable  to characterise \textit{minimal closed $\Gamma$-invariant} sets, that is, closed $\Gamma$-invariant sets  not containing other closed $\Gamma$-invariant sets except the void set and itself.

\subsection{Unboundedness hypothesis}
The last of the fundamental hypotheses of our paper:
\begin{itemize}
	\item[($\mathfrak{U}$)] \textbf{Unboundedness of the action.}  For every $x\in\mathbb R$ we have $g_1(x)<x< g_2(x)$ for some $g_1, g_2\in \Gamma$
\end{itemize}	
guarantees that any closed $\Gamma$-invariant set  is unbounded. In fact, we have the following easy lemma:

\begin{lem}\label{lem:unbound}
Hypothesis ($\mathfrak{U}$) is satisfied if and only if for any $x\in \R$,
$$\sup_{g\in\Gamma^*} g(x)=+\infty\quad \mbox{ and }\quad \inf_{g\in\Gamma^*} g(x) =-\infty. $$
In particular, if condition ($\mathfrak{U}$) is satisfied then any nonempty $\Gamma$-invariant set is unbounded on both sides.

\end{lem}
\begin{proof}
Suppose first that ($\mathfrak{U}$) holds and $x_0=\sup_{g\in\Gamma^*} g(x) <\infty$. Then for all $g_0\in \Gamma$
$$
g_0(x_0)= \sup_{g\in\Gamma^*} (g_0 g)(x) \leq \sup_{g\in\Gamma^*} g(x)= x_0,
$$
which  contradicts to ($\mathfrak{U}$).

Conversely, assume that there is an $x\in \R$ such that  $g_1(x)\leq x$ for all $g_1\in\Gamma$. Since all the homeomorphisms preserve the order,  $g_2 (g_1(x)) \le g_2(x) \le x$ for all $g_1,g_2\in \Gamma$. Thus the induction argument yields $g(x)\le x$ for all $g\in\Gamma^n$ and $n\in\mathbb N$. This finally implies that
 $\sup_{g\in\Gamma^*}g(x)\leq x$.

\end{proof}
In particular, under condition  ($\mathfrak{U}$)  the support of any invariant measure is unbounded in both direction. Note  also that if ($\mathfrak{U}$) holds for $\Gamma$ it also holds for $\Gamma^{-1}$.

\subsection{Ergodic measures and Ratio Ergodic Theorem}\label{sec: Ratio erg}

Among $\mu$-invariant measures, ergodic measures play a special role. We present here the main facts and we refer to Section \ref{sec:ergodic} for a more detailed discussion.

For any measurable $A\subseteq \R$ denote by  $\nu_A$ the restriction of $\nu$ to $A$. The restriction is called
trivial if  either $\nu(A)=0$ or $\nu(\R\setminus A)=0$.
We say that a  measure $\nu$ is {\it ergodic} if for  any $A\in {\mathcal B}(\R)$ such that the restriction $\nu_A$ is invariant, it must be also trivial.
 In our setting we can say that if an invariant measure is ergodic, then  any closed $\Gamma$-invariant set $M$ is either null or has full measure: $\nu(M)=0$ or $\nu(\R\setminus M)=0.$
In  Section \ref{sec: appendix erg mes} we give a~more detailed discussion of other equivalent characterisations of the ergodic measures.

Ergodic measures can be seen as atomic bricks that are used to construct any invariant measure. In fact, any invariant measure $\nu$ can be decomposed into ergodic components, in the sense that there exists a measurable set $\EE_\nu$ of  ergodic measures and  a finite measure $\eta_\nu$ on $\EE_\nu$ such that
\begin{equation}\label{e1_20.02.13}
\nu(f)=\int_{\EE_\nu}\nu_e(f) d\eta_\nu(e) \quad\mbox{for all $f\in C_C(\R)$}.
\end{equation}
In  Theorem \ref{Thm: erg dec} we provide a proof of this decomposition for conservative  Markov--Feller processes. Note that the above decomposition entails  that if there are two different invariant measures, there must exist at least two different ergodic measures.
Another consequence is that if $\nu$ is invariant, there exists an ergodic measure  $\nu_e$ such that $\supp\nu_e\subseteq \supp\nu$. In fact, for $\eta_\nu$-almost all $e\in \EE_\nu$ we have $\nu_e(\R\setminus\supp\nu)=0$. Hence we have $\supp\nu_e\subseteq \supp\nu$.

A fundamental property of ergodic $\mu$-invariant Radon measures, that we will often use in the sequel, is  the Ratio Ergodic Theorem (or the Chacon-Ornstein Theorem), that gives the asymptotic behaviour of the partial sum defined by:
 \begin{equation}\label{eq:sn}
S_n\phi(x):=\phi (\g_{ n}\cdots \g_1(x))+\ldots+  \phi(\g_{1}(x))+\phi(x)=\sum_{k=0}^n\phi(X_n^x)
\end{equation}
for any measurable function $\phi\in L^1(\R,\nu)$ and $x\in \R$.
Observe that if $\phi$ is the indicator function of some set $A$, then $S_n\phi(x)=S_n{\bf 1}_A(x)$ is the number of visits in $A$ up to time $n$ for the Markov chain $(X_n^x)_{n\in\mathbb N}$ starting at $x$.

Whenever recurrence condition ($\mathfrak{R}$) is satisfied, it follows that  for any arbitrary function $\Phi$ whose support contains a recurrent interval $\I$ we have  $S_n\Phi(x)\to +\infty$ for any $x\in\R$, as $n\to\infty$.

If $\nu$ is ergodic for any nonnegative function $\Phi\in L^1(\R,\nu)$  we have $\nu(\Phi)>0$ if and only if  $S_n\Phi(x)\to +\infty$ for $\nu$-almost all $x$
 and in this case the Chacon-Ornstein Theorem \cite{Chacon:Ornstein}  guarantees that for  any $\phi\in L^1(\R, \nu)$  the following limit exists
\begin{equation}\label{eq:chacon}
\lim_{n\to\infty}\frac{S_n\phi(x)}{S_n \Phi(x)} =\frac{\nu(\phi)}{\nu(\Phi)}
\end{equation}
 for $\mu^\N$-almost all sequences $(g_1, g_2,\ldots)\in\Gamma^{\mathbb N}$ and $\nu$-almost all $x\in\R$. This is a consequence of the fact that the shift $\tau$ is a contraction on the space $L^1(\R^\N,\P_\nu)$ and that $\P_\nu$ is ergodic, if $\nu$ is ergodic  (see Section \ref{sec:ergodic} and in  particular Corollary \ref{cor: ergodic shift} for a more complete discussion of these results).

\subsection{Measures with atoms}
The following lemma is useful when we have to deal with some invariant measures $\nu$ that have \textit{atoms}, that is, for which there exists $x\in\R$ such that $\nu(\{x\})>0$. It essentially says that one can have invariant measures with atoms only if  the orbits of action of $\Gamma^{-1}$ are somehow discrete.
\begin{lem}\label{lem-nu-atom}
 Assume that condition (\Rec) is satisfied.
	Let $\nu$ be a $\mu$-invariant Radon measure with atoms and let $K$ be a compact interval that contains the recurrence interval $\I$ and some atoms. Then there exists   $x_0\in K$ such that the orbit $(\Gamma^{-1})^*x_0\cap K$ is finite.
\end{lem}
\begin{proof}
Let $\nu$ be a $\mu$-invariant Radon measure with atoms. We shall abbreviate $\nu (\{x\})$ to $\nu(x)$. Analogously, we shall also write ${\bf 1}_x$ for ${\bf 1}_{\{x\}}$. Note first that because $\nu$ is a Radon measure, there are at most countable many atoms and the mass of all atoms in $K$ is finite. Therefore $\sup_{x\in K}\nu(x)$ is finite and there is $x_K\in K$ such that $\nu(x_K) = \sup_{x\in K}\nu(x)$ (note however that it could be not uniquely determined).
 We will prove  that for any compact set $K$ containing  some atoms of $\nu$:
\begin{equation}\label{eq:ny9}
  \nu(y)=\nu(x_K) \qquad  \mbox{for any } y\in (\Gamma^{-1})^*x_K\cap K.
\end{equation}
Since  the total mass of $K$ is finite this will imply that $(\Gamma^{-1})^*x_K\cap K$ is finite.

Let $O = (\Gamma\cup \Gamma^{-1})^*x_K$ be the orbit of $x_K$ under the action of the group generated by $\Gamma$ endowed with the discrete topology. Note that $O$ is a $\Gamma$ and $\Gamma^{-1}$--invariant countable set.  We can define on $O$ a countable Markov chain $X_n$ (that is just the restriction to $O$ of the Markov chain defined on $\R$) with the transition  kernel
$$
p(x, y)=\P(\g_1(x)=y)=\sum_{g\in\Gamma}{\bf 1}_{y}(g(x))\mu(g) \qquad \forall x,y\in O.
$$
Let  $\overline\nu$ be the measure on $O$ defined by $\overline \nu(x):=\nu(x)$.
Observe that
$\overline \nu = \nu|_{O}$ and $\overline \nu$ remains $\mu$-invariant, i.e.
\begin{equation}\label{eq:ov nu}
  \overline \nu(x) = \sum_{y\in O} p(y, x)\overline \nu(y)
  = \sum_{g\in \Gamma}  \sum_{y\in O}  {\bf 1}_{x}(g(y))  \mu(g) \overline \nu(y)
  = \sum_{g\in \Gamma} \mu(g) \overline \nu(g^{-1}(x)) \quad\text{for $x\in O$}.
\end{equation}
In fact, we have
\begin{align*}
\sum_{g\in \Gamma}  \sum_{y\in O}  {\bf 1}_x(g(y))  \mu(g) \overline \nu(y)&=\sum_{g\in \Gamma}\int_\R {\bf 1}_O(y){\bf 1}_x(g(y)) d\nu(y)\mu(g)\\
&=\sum_{g\in \Gamma}\int_\R {\bf 1}_x(g(y)) d\nu(y)\mu(g)=\nu(x).
\end{align*}
Consider the induced Markov chain on  $O_K = O\cap K$ defined by the kernel
\begin{equation}\label{eq:ny1}
  p_K(x,y): = \P_x(X_{T}=y, T<\infty)= \sum_{n=1}^{\infty}\sum_{\ux\in O_n(x,y)} p(x_1,x_2)\cdots p(x_{n-1},x_n)
  \end{equation} for $x,y\in O_K$,
where $T := \inf\{n\geq 1:\; X_n\in O_K\}$ is the first hitting time of $O_K$ and
$$O_n(x,y):=\{\ux\in O^\N:  x_1=x, x_n=y \mbox{ and } x_i\not\in O_K \mbox{ for all } 1<i<n \}.$$

 Since $K$ contains $\I$, the stopping time $T$ is finite $\mu^{\otimes \mathbb N}$--a.s. for every $x\in O_K$, thus the kernel  is stochastic, i.e.
\begin{equation}\label{eq:ny2}
  p_K(x,O_K)= 1 \qquad\text{ for $x\in O_K$}.
\end{equation}
The restriction of $\overline\nu $ to $O_K$ is  an invariant measure for the Markov kernel $p_K$ (see for instance \cite[Prop. 10.4.6 and Thm. 10.4.7]{MT}), that is
$$\overline{\nu}(y)=\sum_{x\in O_K}p_K(x,y)\overline{\nu}(x)$$
Consider now the reversed Markov chain $\widehat X_n$ on $O$ defined by recursive action of the $\g_i^{-1}$ with the kernel
\begin{equation}\label{eq:p-hat}
\widehat{p}(x,y):=\P(\g_1^{-1}(x)=y)=\P(x=\g_1(y))=p(y,x) \qquad \forall x,y\in O
\end{equation}
 and the induced kernel on $K$:
  \begin{equation*}
 \widehat{p}_K(x,y): = \P_x(\widehat{X}_{\widehat{T}}=y, \widehat{T}<\infty),
 \end{equation*}
 where $\widehat{T} = \inf\{n\geq 1:\; \widehat{X}_n\in O_K\}$.
 Observe also that  by (\ref{eq:ny1}) and (\ref{eq:p-hat}),
  $\widehat{p}_K(x,y) =p_K(y,x)$.
 Thus  for every $n\in \N$ \begin{align*}
 \overline{\nu}(x_K)=\sum_{y\in O_K}\overline{\nu}(y)p_K^n(y,x_K)=\sum_{y\in O_K}\overline{\nu}(y)\widehat{p}_K^n(x_K,y).
 \end{align*}
 Since the kernel $\widehat{p}_K(x,y)$ is sub-stochastic,  $\sum_{y\in O_K}\widehat{p}_K(x,y)=\P_x(\widehat{T}<\infty)\leq 1$ and $\overline{\nu}(y)\leq \overline{\nu}(x_K)$ for $y\in O_K$, it follows that
 $\overline{\nu}(y)= \overline{\nu}(x_K)$, whenever there exists $n$ such that $\widehat{p}_K^n(x_K, y)>0$, that is if $y\in (\Gamma^{-1})^*x_K$. This completes the proof.\end{proof}

\section{Minimality of the supports. Proof of Theorem \ref{thm:main2}}\label{sec: Minimalty}
In this section we are going to prove Theorem \ref{thm:main2}.
We need to consider the reverse random walk with step law given by the probability on $\Hom^+(\R)$ defined as
$$\widehat \mu(g):=\mu(g^{-1})$$ and the associated Feller kernel $$\widehat{P}f(x):=\sum_{g\in\Gamma}f(g^{-1}(x))\mu(g)=\sum_{g\in\Gamma^{-1}}f(g(x))\widehat\mu(g).$$

Theorem \ref{thm:main2}
will be a consequence of  Propositions \ref{prop: case 2}  and \ref{prop: case 1}  that will be stated and proved herein. We will see in the proof that these  two propositions  cover two complementary cases.
The proof of the second proposition shares some  arguments with the paper of Deroin et al. \cite{Deroin:Kleptsyn:Navas:Parwani:2013} on symmetric random walks.

\begin{prop}\label{prop: case 2}
	Let $\nu_1$ and $\nu_2$ be two $\mu$-invariant ergodic measures such that  $\supp\nu_1\subseteq \supp \nu_2$.
	Assume that there are a set $M$ unbounded on both sides and an open interval $J$  having at least two common points with  $\supp \nu_1$ such that for any
	$u\in M$, $$N(\ug, u) = \sup\{n:\; \ell^{-1}_n(u)\in J\}<\infty$$ for $\mu^{\otimes \mathbb N}$-almost all $\ug = (g_1,g_2,\ldots)$.
	Then $\nu_1=C\nu_2$ for some nonzero constant $C>0$.
\end{prop}

\begin{proof}
	Note that to prove the result it is sufficient to ensure that for arbitrary  $a,b\in M$ such that $\I\subset (a,b)$ and any $z\in(a,b)$:
	\begin{equation}\label{eq:p1}
	\frac{\nu_1[a,z)}{\nu_1[a,b)}= \frac{\nu_2[a,z)}{\nu_2[a,b)}.
	\end{equation}
	Indeed, taking the difference we have that for all $a<z_1<z_2<b$
	$$
	\nu_1[z_1,z_2) = C(a,b) \nu_2[z_1,z_2)
	$$
	with $C(a,b) = \nu_1[a,b)/\nu_2[a,b) \in (0,\infty)$.
	Thus $\nu_1$ and $\nu_2$ coincide up to a constant on $[a,b)$. Observe that $\nu_i[a,b)\geq\nu_i(\I)>0$. To extend this equality to the whole line it is sufficient to appeal to unboundedness of $M$. Taking  sequences $(a_n)_{n\in\mathbb N}, (b_n)_{n\in\mathbb N}\subset M$ such that $a_n\to -\infty$ and $b_n\to +\infty$ we deduce the equality of both measures on $\R$.

	\medskip
	
	Now we pass to the proof of \eqref{eq:p1}. Fix $a,b\in M$ such that the recurrent set $\I$ is contained in $(a,b)$. The assumptions  of the proposition assure that there exist two distinct  $y_1,y_2 \in J\cap \supp \nu_1$; we can assume $y_1<y_2$. Taking two sufficiently small neighbourhoods of these point, we can find two disjoint intervals $J_1$ and $J_2$ such that $\nu_1(J_i)>0$ and $J_i\subset J$, $i=1, 2$.
	Note that for any $x_i\in J_i$ ($i=1,2$), $z>a$ and any $n\ge N(\ug) := \max\{ N(\ug,a), N(\ug,b)\}+1$ we have
	\begin{equation}\label{eq:bn1}
	{\bf 1}_{[a,z)}(\ell_n(x_1)) = {\bf 1}_{[\ell_n^{-1}(a), \ell_n^{-1}(z))}(x_1)
	\ge {\bf 1}_{[\ell_n^{-1}(a), \ell_n^{-1}(z))}(x_2) = {\bf 1}_{[a,z)}(\ell_n(x_2))
	\end{equation}
	since $x_1<x_2$ and $\ell_n^{-1}(a)\notin J\supset [x_1,x_2]$. Similarly one can check that
	\begin{equation}\label{eq:bn2}
	{\bf 1}_{[a,b)}(\ell_n(x_1)) = {\bf 1}_{[a,b)}(\ell_n(x_2))
	\end{equation} 	
	using that also
	$ \ell_n^{-1}(b)\notin J$ for appropriately large $n$.
	
	Observe that also $\nu_2(J_i)>0$, because $y_i\in \supp\nu_1\subseteq \supp\nu_2$. By the Chacon--Ornstein Theorem \eqref{eq:chacon}  there exist $x_1\in J_1$ and $x_2\in J_2$  such that for $\mu^{\otimes \mathbb N}$-almost every $(g_i)_{n\in\mathbb N}\in\Gamma^{\mathbb N}$
	\begin{equation}\label{eq:p3}
	\lim_{n\to\infty}\frac{ S_n{\bf 1}_{[a,z)}(x_1)}{ S_n{\bf 1}_{[a,b)}(x_1)}=\frac{\nu_1[a, z)}{\nu_1[a,b)} \quad \mbox{ and } \quad \lim_{n\to\infty}\frac{ S_n{\bf 1}_{[a,z)}(x_2)}{ S_n{\bf 1}_{[a,b)}(x_2)}=\frac{\nu_2[a, z)}{\nu_2[a,b)}.
	\end{equation}
Appealing to the definition of $S_n$ given in \eqref{eq:sn}, formulas \eqref{eq:bn1} and \eqref{eq:bn2} yield for any $n > N(\ug)$
	$$
	S_n{\bf 1}_{[a,z)}(x_1) - S_{N(\ug)}{\bf 1}_{[a,z)}(x_1) \ge S_n{\bf 1}_{[a,z)}(x_2)- S_{N(\ug)}{\bf 1}_{[a,z)}(x_2) $$
	and
	$$
	S_n{\bf 1}_{[a,b)}(x_1) - S_{N(\ug)}{\bf 1}_{[a,b)}(x_1) = S_n{\bf 1}_{[a,b)}(x_2)- S_{N(\ug)}{\bf 1}_{[a,b)}(x_2).
	$$
	Recall that since the recurrent set $\I$ is a subset of $(a,b)$, $S_n{\bf 1}_{[a,b)}(x_i)\to \infty$ for $i=1,2$ $\mu^{\otimes \mathbb N}$--a.s.
	Hence on a set of probability 1 we have
	
	\begin{align*}
	\frac{\nu_1[a,z)}{\nu_1[a,b)} &=
	\lim_{n\to\infty}\frac{ S_n{\bf 1}_{[a,z)}(x_1)}{ S_n{\bf 1}_{[a,b)}(x_1)}\\
&= \lim_{n\to\infty}\frac{ S_{N(\ug)}{\bf 1}_{[a,z)}(x_1)}{ S_n{\bf 1}_{[a,b)}(x_1)} +
	\frac{ S_n{\bf 1}_{[a,z)}(x_1)- S_{N(\ug)}{\bf 1}_{[a,z)}(x_1)}{ S_n{\bf 1}_{[a,b)}(x_1)- S_{N(\ug)}{\bf 1}_{[a,b)}(x_1)}
	\cdot \frac{ S_n{\bf 1}_{[a,b)}(x_1)- S_{N(\ug)}{\bf 1}_{[a,b)}(x_1)}{ S_n{\bf 1}_{[a,b)}(x_1)}\\
	&\ge  \lim_{n\to\infty}\frac{ S_n{\bf 1}_{[a,z)}(x_2)- S_{N(\ug)}{\bf 1}_{[a,z)}(x_2)}{ S_n{\bf 1}_{[a,b)}(x_2)- S_{N(\ug)}{\bf 1}_{[a,b)}(x_2)}
=\lim_{n\to\infty}\frac{ S_n{\bf 1}_{[a, z)}(x_2)}{ S_n{\bf 1}_{[a, b)}(x_2)}
	= \frac{\nu_2[a,z)}{\nu_2[a,b)},
	\end{align*}
	the penultimate equality by the fact that $S_n{\bf 1}_{[a, b)}(x_2)\to\infty$ as $n\to\infty$.
	Interchanging in \eqref{eq:p3} the role of measures $\nu_1$ and $\nu_2$, that is choosing   $x_1\in J_1$ and $x_2\in J_2$ such that
	\begin{equation*}
	\lim_{n\to\infty}\frac{ S_n{\bf 1}_{[a,z)}(x_2)}{ S_n{\bf 1}_{[a,b)}(x_2)}=\frac{\nu_1[a, z)}{\nu_1[a,b)} \quad \mbox{ and } \quad \lim_{n\to\infty}\frac{ S_n{\bf 1}_{[a,z)}(x_1)}{ S_n{\bf 1}_{[a,b)}(x_1)}=\frac{\nu_2[a, z)}{\nu_2[a,b)},
	\end{equation*}we arrive at the converse inequality
	$$\frac{\nu_1[a,z)}{\nu_1[a,b)}
	\le \frac{\nu_2[a,z)}{\nu_2[a,b)},
	$$
	concluding thus \eqref{eq:p1}. This completes the proof.
\end{proof}

\begin{prop}\label{prop: case 1}
	Let $\nu_1$ and $\nu_2$ be two ergodic invariant measures such that $\supp\nu_1\subseteq \supp\nu_2$  and  $\nu_1$ has no atoms.
	Suppose that there exists a $\widehat\mu$-invariant Radon measure $\widehat \nu$ such that $\supp\nu_1\subseteq \supp\widehat\nu$. Then $\nu_1=C\nu_2$ for some positive constant $C$.
\end{prop}

The existence of the measure $\widehat{\nu}$ enables us to assure  that the number of visits to a given interval of processes
$(X_n^x)_{n\in\mathbb N}$ and $(X_n^y)_{n\in\mathbb N}$ starting from two different points $x$ and $y$ does not differ too much if $x$ and $y$ are close enough. Our arguments are partially  inspired by the techniques introduced  in  \cite{Deroin:Kleptsyn:Navas:Parwani:2013}.

\begin{lem}
	\label{lem: appr nu hat } Assume that (\Rec) is satisfied.
	Let $\nu$ be an ergodic $\mu$-invariant measure and let $\widehat\nu$ be a $\widehat\mu$-invariant Radon measure. Let  $a$ and $b$ be  two points of the support of $\widehat\nu$ such that  $\nu[a,b)>0$.
	Fix  two  constants $p,\varepsilon \in (0,1)$ and let $\delta = \min\{\widehat\nu(I_{a,\varepsilon}),\widehat\nu(I_{b,\varepsilon})\}>0$, where $I_{c,\varepsilon}:= (c-\varepsilon ,c+\varepsilon)$ for aritrary $c\in\mathbb R$. Then for $\nu$-a.e. $y$ and any $x<y$ satisfying $\widehat\nu[x,y)<(1-p)\delta$,
	$$
	\mu^{\otimes \mathbb N}\left(\bigg\{\ug:\; \uplim_{n\to\infty} \left|\frac{S_n{\bf 1}_{[a,b)}(x)}{S_n{\bf 1}_{[a,b)}(y)}-1\right|\leq \frac{\nu(I_{a,\varepsilon})+ \nu(I_{b,\varepsilon})}{\nu[a,b)}
	\bigg\}\right)\ge p
	.$$
\end{lem}
\begin{proof}
	We start with an  observation that if two points $x$ and $y$ are close with respect to the distance measured by $\widehat\nu$, i.e. if
	$\widehat\nu[x,y)<(1-p)\delta$, then with probability at least $p$ the distance between two trajectories $(X_n^x)_{n\in\mathbb N}$ and $(X_n^y)_{n\in\mathbb N}$ remains small, i.e.,
	\begin{equation}\label{eq:small}
	\P(\lim_{n\to\infty}\widehat\nu[X_n^x, X_n^y) < \delta)=\mu^{\otimes \mathbb N}(\big\{\ug:\;  \lim_{n\to\infty}\widehat\nu[X_n^x, X_n^y) < \delta \big\}) \ge p.
	\end{equation}
	This fact was already  proved in \cite[Lemma 6.6]{Deroin:Kleptsyn:Navas:Parwani:2013}, nevertheless for the reader convenience we present here a complete argument. Note first that since the measure $\widehat\nu$ is $\widehat\mu$-invariant, the sequence $\widehat\nu[X_n^x,X_n^y)$ forms a positive martingale, thus by the Martingale Convergence Theorem, it converges to a nonnegative random variable $v(x,y)$. Fatou's Lemma entails
	$$
	\E v(x,y) \le \lim_{n\to\infty} \E\big[ \widehat\nu[X_n^x,X_n^y)\big] = \widehat\nu[x,y)
	$$ and, finally, by the Markov Inequality we obtain
	$$\P(\big\{ v(x,y) >\delta \big\}) \le \P\left(\bigg\{ v(x,y)  > \frac{\widehat\nu[x, y)}{1-p}\bigg\}\right) \le \frac{(1-p) \E v(x,y)}{\widehat\nu[x,y)} \le 1-p
	$$ completing thus the proof of \eqref{eq:small}.
	
	To proceed further we need an additional auxiliary inequality. Namely, note that for any $x<y$ we have
	\begin{align*}
	\left|{\bf 1}_{[a ,b)}(x)-{\bf 1}_{[a ,b)}(y)\right|
	&={\bf 1}_{[a ,b)}(y){\bf 1}_{(-\infty,a )}(x)+ {\bf 1}_{[a ,b)}(x){\bf 1}_{[b,+\infty)}(y)\\
	&\leq {\bf 1}_{I_{a,\varepsilon}}(y)+{\bf 1}_{\{[x,y)\supseteq  I_{a,\varepsilon}\}}(y)+ {\bf 1}_{I_{b,\varepsilon}}(y)+{\bf 1}_{\{[x,y)\supseteq  I_{b,\varepsilon}\}}(y)\\
	&\leq  {\bf 1}_{I_{a,\varepsilon}}(y)+{\bf 1}_{I_{b,\varepsilon}}(y)+{\bf 1}_{\{\widehat\nu[x,y)\geq \widehat\nu(I_{a,\varepsilon})\}}(y)
	+{\bf 1}_{\{\widehat\nu[x,y)\geq  \widehat\nu(I_{b,\varepsilon})\}}(y)\\
	&\leq {\bf 1}_{I_{a,\varepsilon}}(y)+ {\bf 1}_{I_{b,\varepsilon}}(y)+2\cdot {\bf 1}_{\{\widehat\nu[x,y)\geq \delta
		\}}(y).	\end{align*}
	Replacing  $x,y$ by $\ell_k(x)$ and $\ell_k(y)$ respectively and next summing over $k$, we obtain for any $x<y$ and $n\ge 0$
	\begin{equation}\label{eq:s1}
	\begin{split}
	\big|S_n{\bf 1}_{[a ,b)}(x) & -S_n{\bf 1}_{[a ,b)}(y)\big| \\
	&\leq S_n{\bf 1}_{I_{a,\varepsilon}}(y)+S_n{\bf 1}_{I_{b,\varepsilon}}(y)+ 2\mathrm{card}\{k\leq n : \widehat\nu[X_k^x,X_k^y)\geq \delta \}.
	\end{split}
	\end{equation}
	Since $\nu[a ,b)>0$, the Chacon-Ornstein Theorem (\refeq{eq:chacon}) entails
	\begin{equation}\label{eq:s3}
	\lim_{n\to\infty}\frac{S_n{\bf 1}_{I_{a,\varepsilon}}(y)+S_n{\bf 1}_{I_{b,\varepsilon}} (y)}{S_n{\bf 1}_{[a ,b)}(y)}= \frac{\nu(I_{a,\varepsilon})+ \nu(I_{b,\varepsilon})}{\nu[a ,b)}
	\end{equation}
	for $\nu$--a.e. $y$.
	Furthermore, for $\nu$--a.e. $y$, $S_n{\bf 1}_{[a ,b)}(y)$ converges  to $+\infty$, since $\nu[a,b)>0$.   Now, fix a $y$ for which the above limit exists and take arbitrary $x<y$ such that $\widehat\nu[x,y]<(1-p)\delta$. Then,
	in view of \eqref{eq:small}, on a set of probability at least $p$  we have $\lim_{n\to \infty }\widehat\nu [X_n^x,X_n^y) < \delta$. Thus invoking
	\eqref{eq:s1} on the intersection of this set with the set of full measure for which \eqref{eq:s3} hold we obtain
	\begin{align*}
	\limsup_{n\to\infty}\left|\frac{S_n{\bf 1}_{[a ,b)}(y)-S_n{\bf 1}_{[a ,b)}(x)}{S_n{\bf 1}_{[a ,b)}(y)}\right|
	&\leq \lim_{n\to\infty}
	\frac{S_n{\bf 1}_{I_{a,\varepsilon}}(y)+S_n{\bf 1}_{I_{b,\varepsilon}}(y)
	}{S_n{\bf 1}_{[a ,b)}(y)}\\
	&= \frac{\nu(I_{a,\varepsilon})+ \nu(I_{b,\varepsilon})}{\nu[a ,b)}
	\end{align*}
	and the proof of the lemma is completed.
\end{proof}

\begin{proof}[Proof of Proposition \ref{prop: case 1}]
	 Now we are going to prove that for any $a,b\in \supp\nu_1 \subseteq \supp \widehat \nu$ such that $\nu_1[a,b)>0$ and $\nu_2[a,b)>0$ and for any $z\in(a,b]$	\begin{equation}\label{eq:sb1}
	 \frac{\nu_1[a,z)}{\nu_1[a,b)}=\frac{\nu_2[a,z)}{\nu_2[a,b)}.
	 \end{equation}
	  The desired result will be shown by the same argument as in the proof of Proposition \ref{prop: case 2}, using the fact that $\supp\nu_1$ is unbounded.
	
	{\sc Step 1.}  First we will prove that \eqref{eq:sb1} holds
for $z\in  \supp \widehat \nu $  such that $\nu_1[a,z)>0$.
	
	Fix $p\in(1/2,1)$, choose $\varepsilon>0$ such that the intervals $I_{a,\varepsilon}$, $I_{b,\varepsilon}$ and $I_{z,\varepsilon}$ are pairwise disjoint and put
	$$
	\delta:=\min\{\widehat\nu(I_{a,\varepsilon}),\widehat\nu(I_{b,\varepsilon}),\widehat\nu(I_{z,\varepsilon}) \}>0.$$
	We claim that there  exist two disjoint open  intervals $I_1, I_2$ and an  interval $I_0 \supset I_1\cup I_2$ such that
	$$ \sup\{x\in I_2\}\leq\inf\{x\in I_1\}, \quad \nu_1(I_1)>0,\quad \nu_1(I_2)>0\mbox{ and }
	\widehat\nu(I_0)<(1-p)\delta.$$
	In fact, let $\I$ be an open interval such that $\supp\nu_1\cap\I\not=\emptyset$.
	Since $\nu_1$ has no atoms $\supp\nu_1\cap \I$ contains infinitely many points. Thus there exists a strictly monotone sequence $z_n\in \supp\nu_1\cap \I$. Suppose that $z_n$ is increasing (the decreasing case is similar). Consider the open neighbourhood of $z_n$ defined by  $J_n:=(\frac{z_n+z_{n-1}}{2}, \frac{z_n+z_{n+1}}{2})$. Intervals $J'_n=[\frac{z_n+z_{n-1}}{2}, \frac{z_n+z_{n+1}}{2})$ for $n\in\mathbb N$  are disjoint and contained in the bounded interval $\I$, whence $\widehat\nu(J'_n)$ converges to 0 and thus $\widehat\nu(J_n)<(1-p)\delta/2$ for any sufficiently large~$n$. We take $I_2=J_{n}$, $I_1=J_{n+1}$ and $I_0=J'_n\cup J'_{n+1}$.
	
	Since $\nu_1(I_1)>0$ and $\nu_1[a,b)>\nu_1[a,z)>0$, by Chacon-Orstein's Theorem (\ref{eq:chacon}) and appealing twice to  Lemma \ref{lem: appr nu hat } (first for points  $a,b$ and then for $a,z$) we deduce that there exists  $x_1\in I_1$ such that  $\mu^{\otimes \mathbb N}$--almost surely:
	\begin{equation}\label{eq:sb5}		
	\lim_{n\to\infty}\frac{S_n{\bf 1}_{[a ,z)}(x_1)}
	{S_n{\bf 1}_{[a ,b)}(x_1)}= \frac{\nu_1[a ,z)}{\nu_1[a ,b)}
	\end{equation}
	and for all $x_2\in I_2$ with probability greater than $1-2(1-p)=2p-1>0$,	
	\begin{align*}
	\limsup_{n\to\infty} \left|\frac{S_n{\bf 1}_{[a,b)}(x_2)}{S_n{\bf 1}_{[a,b)}(x_1)}-1\right|&\leq \frac{\nu(I_{a,\varepsilon})+ \nu(I_{b,\varepsilon})}{\nu[a,b)},\\ 
\limsup_{n\to\infty} \left|\frac{S_n{\bf 1}_{[a,z)}(x_2)}{S_n{\bf 1}_{[a,z)}(x_1)}-1\right|&\leq \frac{\nu(I_{a,\varepsilon})+ \nu(I_{b,\varepsilon})}{\nu[a,z)}.
	\end{align*}
	
	Now since $\nu_2(I_2)>0$ and $\nu_2[a,b)>0$ we can chose $x_2\in I_2$ such that $\mu^{\otimes \mathbb N}$--almost surely: 	
	\begin{equation}\label{eq:sb6}  \lim_{n\to\infty}\frac{S_n{\bf 1}_{(a ,z)}(x_2)}{S_n{\bf 1}_{(a ,b)}(x_2)}= \frac{\nu_2[a ,z)}{\nu_2[a ,b)}.
	\end{equation}
	Thus, with $\mu^{\otimes \mathbb N}$ probability  at least $2p-1>0$, we can write
	\begin{multline*}
	\bigg| \frac{\nu_1[a ,z)}{\nu_1[a ,b)}\times\frac{\nu_2[a ,b)}
	{\nu_2[a ,z)}-1\bigg|=
	\lim_{n\to\infty}\left|
	\frac{S_n{\bf 1}_{[a ,b)}(x_2)}{S_n{\bf 1}_{[a ,b)}(x_1)}:
	\frac{S_n{\bf 1}_{[a ,z)}(x_2)}{S_n{\bf 1}_{[a ,z)}(x_1)}-1\right|\\
	\leq  \bigg(\frac{\nu_1(I_{a,\varepsilon})+ \nu_1(I_{b,\varepsilon})}{\nu_1[a ,b)} +  \frac{\nu_1(I_{a,\varepsilon})+\nu_1(I_{z,\varepsilon})}{\nu_1[a ,z)}\bigg) : {\bigg(1- \frac{\nu_1(I_{a,\varepsilon})+\nu_1(I_{z,\varepsilon})}{\nu_1[a ,z)}\bigg)},
	\end{multline*}
	where for the last inequality we used the inequality $\left|\frac{1+\epsilon_n}{1+\eta_n}-1\right|\leq \frac{|\epsilon_n| + |\eta_n| }{1-|\eta_n| }$.
	Since the measure $\nu_1$ is atomless, sending $\varepsilon$ to $0$ in the last estimates proves \eqref{eq:sb1}.
	
	\medskip

	{\sc Step 2.}
	 Now we are going to prove that (\ref{eq:sb1}) holds for any $a,b\in \supp\nu_1$ and any $z\in(a,b]$.
	Let   \begin{align*}
	\overline{z}&:= \min\{x:\; x\in \supp\nu_1 \cap [z,+\infty)\}\in\supp\nu_1,\\
	\underline{z}&:= \max\{x:\; x\in \supp\nu_1 \cap (-\infty,z]\}\in\supp\nu_1.
	\end{align*}
	 In particular, since $\nu_1$ has no atoms, $(\underline{z}, \overline{z})\cap\supp\nu_1=\emptyset$. For all $c < \underline{z}\leq z$
we have
	\begin{align}\label{eq: over z}
	\nu_1[c,\overline{z})&=\nu_1[c,z)+\nu_1[z,\overline{z})=\nu_1[c,z)+\nu_1(z,\overline{z})=\nu_1[c,z)\\\label{eq: under z}
	\nu_1[c,\underline{z})&=\nu_1[c,\underline{z}) +\nu_1(\underline{z},z)=\nu_1[c,\underline{z}) +\nu_1[\underline{z},z) =\nu_1[c,z).
	\end{align}
	Since (\UB) holds there exists $a_0\in \supp\nu_1$ such that $\nu_1[a_0,a)>0$ and $\nu_2[a_0,a)>0$, by the fact that $\supp\nu_1\subset\supp\nu_2$. Since $\supp\nu_1\subseteq \supp\widehat{\nu}$, by step 1 for any $z\in(a,b]\cap\supp\nu_1$ we have
	\begin{equation}\label{eq: nu1 ov nu2 aux}
	\nu_1[a_0,\underline{z})=C \nu_2[a_0,\underline{z})\quad \mbox{ and } \nu_1[a_0,\overline{z})=C \nu_2[a_0,\overline{z}),
	\end{equation}
	with $C:=\nu_1[a_0,b)/\nu_2[a_0,b)\in(0,\infty)$. Oberving that $a\leq \underline{z}\leq \overline{z}\leq b$ and applying (\ref{eq: over z}), (\ref{eq: under z}) and (\ref{eq: nu1 ov nu2 aux}) one obtains
	$$\nu_1[a,z)=\nu_1[a,\overline{z})=\nu_1[a_0,\overline{z})-\nu_1[a_0,a)=C\nu_2[a_0,\overline{z})-C\nu_2[a_0,a)=C\nu_2[a,\overline{z})\geq C\nu_2[a,z)$$
	and
	$$\nu_1[a,z)=\nu_1[a,\underline{z})=\nu_1[a_0,\underline{z})-\nu_1[a_0,a)=C\nu_2[a_0,\underline{z})-C\nu_2[a_0,a)=C\nu_2[a,\underline{z})\leq C\nu_2[a,z).$$
	Thus $\nu_1[a,z)= C\nu_2[a,z)$ and  \eqref{eq:sb1} follows taking the quotient. The proof is complete.
\end{proof}

\begin{proof}[Proof of Theorem \ref{thm:main2}]
Proof of (1). Let $\nu_1$ and $\nu_2$ be two ergodic invariant Radon measures such that $\supp\nu_1\subseteq\supp\nu_2$ and $\supp\nu_1$ is not discrete in $\R$.

We will consider the two following complementary cases:
\begin{itemize}
	\item[{(\caseA)}] there exists  an open interval $J$  having at least two common points with  $\supp \nu_1$ such that $$C_J:=\left\{x\in\R\left|\,\,\sum_{k=0}^{\infty}\widehat{P}^k{\bf 1}_J(x)<\infty\right.\right\}$$
	is not empty.
	\item[{ (\caseB)}]  For all open intervals $J\subset \R$  having at least two common points with  $\supp \nu_1$ we have:
	$$\sum_{k=0}^{\infty}\widehat{P}^k{\bf 1}_J(x)= \infty \quad \forall x\in\R.$$
\end{itemize}

The theorem is a consequence of  Proposition \ref{prop: case 2} for case (\caseA) and Proposition \ref{prop: case 1} for case (\caseB).

\textbf{Case (\caseA)}. We claim that, in this case, the set $C_J$ is unbounded on both sides and  for any
$u\in C_J$,\begin{equation}\label{eq:s9}
N(\ug, x):=\sup\{n : g_1^{-1}\cdots g_n^{-1} (x)\in J \}<\infty
\end{equation} for $\mu^{\otimes \mathbb N}$--a.e. $\ug = (g_1,g_2,\ldots)$.
Then the fact that $\nu_1=C\nu_2$ is a consequence of Proposition \ref{prop: case 2}, with $M=C_J$.

To prove the claim observe that
$$	\sum_{k=0}^{\infty}\widehat{P}^k{\bf 1}_J(x)
=\E\left[\sum_{n=0}^{\infty}{\bf 1}_J(\g_1^{-1}\cdots \g_n^{-1} (x))\right]
=\E(\mathrm{card}\{n: \g_1^{-1}\cdots \g_n^{-1} (x)\in J\}).
$$
In particular, for $x\in C_J$   the sequence $ \g_1^{-1}\cdots \g_n^{-1} (x)$ visits $J$  finitely many times with probability 1, that is $N(\ug, x)<\infty$ $\mu^{\otimes \mathbb N}$--a.s.

Observe also that the set $C_J$ is $\Gamma^{-1}$--invariant. In fact, if $x\in C_J$ and $g_0\in \Gamma$, then
\begin{multline*}
\infty>\sum_{k=0}^{\infty}\widehat{P}^k{\bf 1}_J(x)\geq \sum_{k=1}^{\infty}\widehat{P}^k{\bf 1}_J(x)
= \sum_{g\in\Gamma}\sum_{k=0}^{\infty}\widehat{P}^k{\bf 1}_J(g^{-1}x)\mu(g)
\geq\sum_{k=0}^{\infty}\widehat{P}^k{\bf 1}_J(g_0^{-1}x)\mu(g_0).
\end{multline*}
Thus $g_0^{-1}x\in C_J$,  since $\mu(g_0)>0$. In particular, since (\UB) holds also for $\Gamma^{-1}$,  Lemma \ref{lem:unbound} entails that $C_J$ is unbounded.

\textbf{Case (\caseB)}. 	We are going to prove that under condition (\caseB) $\nu_1$ has no atoms and there exists a $\widehat{\mu}$-invariant Radon measure $\widehat{\nu}$ such that $\supp\nu_1\subseteq \supp\widehat{\nu}$.
Then the fact that $\nu_1$ is a multiple of $\nu_2$  follows from Proposition \ref{prop: case 1}.

We first prove that $\nu_1$ cannot have atoms. Let $K$ be a compact set that contains the recurrence interval and an accumulation point of $\supp\nu_1$. If $\nu_1$ had an atom in $K$ then, according to Lemma \ref{lem-nu-atom}, there would exist an $x_0\in K$ such that its $\Gamma^{-1}$--orbit
$M_0\,(= (\Gamma^{-1})^*x_0)$
has a finite number of points in $K$. But since $K$ contains an accumulation point of $\supp\nu_1$, there exists an  interval $J\subset M_0^c$ that contains at least two distinct points  $y_1$ and $y_2$ of $\supp\nu_1$. By $\Gamma^{-1}$--invariance of $M_0$ we deduce that for any $x\in M_0$,
$ g^{-1}x\notin J$ for $g\in \Gamma^*$. Thus  $M_0\subset C_J\not=\emptyset$, which leads to a contradiction.

	Since there exists a compact interval $J$ such that $C_J=\emptyset$, the Feller kernel $$\widehat{P}f(x):=\sum_{g\in\Gamma}f(g^{-1}(x))\mu(g)=\sum_{g\in\Gamma^{-1}}f(g(x))\widehat\mu(g)$$
	 is topologically conservative and therefore it has at least one invariant Radon measure $\widehat \nu$ (see Lin's Theorem \cite[Theorem 5.1]{Lin70}).
	 The set  $M_0:=\supp\widehat \nu$ is then closed and $\Gamma^{-1}$-invariant. Suppose now that there exists $y\in \supp \nu_1$ but  $y \not \in \supp\widehat{\nu}$. Since $\supp\widehat{\nu}$ is closed and $\nu_1$ has no atoms there exists then $J\subseteq\R\setminus \supp\widehat{\nu}$ that contains at least two distinct points  $y_1$ and $y_2$ of $\supp\nu_1$. By $\Gamma^{-1}$-invariance of $\supp\widehat{\nu}$ we conclude as above that for any $x\in \supp\widehat{\nu}$,
	 $ g^{-1}x\notin J$ for any  $g\in \Gamma^*$, that is  $\supp\widehat{\nu}\subset C_J\not=\emptyset$. Which leads to a contradiction.

	To  prove (2)  take
	   $\nu$ to be an ergodic invariant measure   and suppose  it does not contain  a $\Gamma$--invariant discrete set. Let $M\subseteq \supp\nu$ be a non empty closed $\Gamma$-invariant set. Then, by recurrence, there exists an ergodic invariant measure $\nu_1$ such that $\supp\nu_1\subseteq M$.  If $\supp\nu$ does not contain a discrete set, then we can apply the first part of the theorem obtaining that $\supp\nu_1=\supp\nu$. Hence $M=\supp\nu$. This proves that $\supp\nu$ is minimal.
	
Conversely, to prove (3), take a minimal closed $\Gamma$-invariant set $M$,  by recurrence. Then there exists an ergodic invariant measure $\nu_1$ such that $\supp\nu_1\subseteq M$. By minimality of $M$ we have $\supp\nu_1= M$. Take now another ergodic measure $\nu_2$ such that $\supp\nu_2=M=\supp\nu_1$. If $M$ is not discrete we can apply the first part of the theorem to conclude that $\nu_1$ and $\nu_2$ coincide up to a multiplicative constant. If $M$ is discrete  observe that any $x\in M$ is an atom for both $\nu_1$ and $\nu_2$, thus, invoking the Chacon-Ornstein theorem, we obtain
	that for any bounded function $\phi$  with compact support and for all $x\in M$ it holds
	$$\frac{\nu_1(\phi)}{\nu_1(\Phi)}=\lim_{n\to\infty}\frac{S_n\phi(x)}{S_n \Phi(x)} \qquad \mbox{ and } \qquad \frac{\nu_2(\phi)}{\nu_2(\Phi)}=\lim_{n\to\infty}\frac{S_n\phi(x)}{S_n \Phi(x)}.$$
From this we finally obtain that $\nu_1=C\nu_2$ with $C= \frac{\nu_1(\Phi)}{\nu_2(\Phi)}$. This completes the proof of~(3).
\end{proof}

\section{Uniqueness of an invariant measure: Proof of Theorem \ref{thm:main}} \label{sec: uniqueness}

%

We start with the following lemma:

\begin{lem}\label{lem:two invariant sets} Assume  that hypotheses (\Cont) and (\UB) hold. Then
any two nonempty and closed  $\Gamma$--invariant sets $M_1$ and $M_2$
have nonempty intersection, i.e. $M_1\cap M_2\neq \emptyset$.
\end{lem}
\begin{proof}
Assume in contrary that $M_1\cap M_2=\emptyset$ for some $\Gamma$--invariant sets $M_1$ and $M_2$. Consider the class of all open intervals such that
	$$\mathcal{J}=\{J=(a,b): a\in M_1, b\in M_2\mbox{ and } (a,b)\subset (M_1\cup M_2)^c\}.$$
	Observe that all the intervals belonging to  $\mathcal{J}$ are disjoint. Furthermore, note  that if the sets  $M_1$ and $M_2$ are disjoint, then  for all pairs
$m_1\in M_1$, $m_2\in M_2$ such that $m_1<m_2$ there exists $J=(a,b)\in \mathcal{J}$ such that $J\subset (m_1,m_2)$. Indeed,  one can just take $a=\sup\{m\in M_1| \,m\leq m_2\}$ and $b= \inf\{m\in M_2|\, m\geq a\}$.
	
	Let $\I$ be the compact interval that appears  in (\Cont). We claim that there are only  finitely many intervals $J\in \mathcal{J}$ which are subsets of $\I$. Indeed, suppose that there are infinitely many elements $J_i=(a_i, b_i)$ of $\mathcal{J}$ such that $J_i\subset \I$ for $i\in \N$. Since both sequences $\{a_i\}_{i\in \N}$ and $\{b_i\}_{i\in \N}$ are contained in the compact interval $\I$, there exist a subsequence $\{i_k\}_{k\in\N}$ of $\N$ such that sequences $\{a_{i_k}\}_{k\in\N}$ and $\{b_{i_k}\}_{k\in\N}$ are convergent. We denote by $a_0$ and $b_0$ their corresponding limits. Recalling that both sets $M_i$ are closed, we deduce that $a_0\in M_1$ and $b_0\in M_2$.    On the other hand,  since all the intervals $J_{i_k}$ are disjoint and contained in a compact set $\I$ their diameters $|b_{i_k}-a_{i_k}|$ converge to zero. Thus $a_0=b_0\in M_1\cap M_2$ and we obtain a contradiction.
	
Denote by	$J_1,\ldots J_N$  all the disjoint intervals, elements of $\mathcal{J}$, contained in $\I$.
In view of Lemma \ref{lem:unbound}, since the  sets $M_1$ and $M_2$ are $\Gamma$--invariant, they are unbounded. Thus there exists an additional interval $J_{N+1}\in\mathcal{J}$ disjoint with $\I$ and all the remaining chosen intervals $J_i$ for $i\leq N$. Condition (\Cont) entails the existence of $g\in \Gamma^*$ such that  $g(J_1\cup\cdots\cup J_N\cup J_{N+1})\subset \I$. Since $g$ is a homeomorphism preserving the order,  for every $i\leq N+1$ it maps intervals  $J_i=(a_i,b_i)$ onto open intervals $g(J_i)=(g(a_i),g(b_i))$. Observe also that $g(a_i)\in M_1\cap \I$ and $g(b_i)\in M_2\cap \I$, thus for every $i\in\{1,\ldots ,N+1\} $ there exists $j_i\in  \{1,\ldots ,N\} $ such that $g(J_i)\supseteq J_{j_i}$ and then the pigeonhole principle entails that $j_{i_1}=j_{i_2}$ for some $i_1\not=i_2$. This means that both
$g(J_{i_1})$ and $g(J_{i_2})$ contain $J_{j_1}$ and therefore cannot  be disjoint. Moreover, $J_{i_1}\cap J_{i_2}\supset g^{-1}(J_{j_1})\not=\emptyset$ contradicting thus to the choice of the intervals  $J_{i_1}$ and $J_{i_2}$ as disjoint sets. So, we finally arrive at the conclusion that two closed and $\Gamma$--invariant  sets $M_1$ and $M_2$ must have a nonempty intersection. This completes the proof.
\end{proof}

\begin{proof}[Proof of Theorem \ref{thm:main}]
Suppose that there exist   two different invariant Radon measures. Without loss of generality, using ergodic decomposition, we may assume that there exist two different ergodic Radon measures $\tilde\nu_1$ and $\tilde\nu_2$.
We claim then that  there are two different invariant ergodic Radon measures $\nu_1$ and $\nu_2$ such that $\supp\nu_1\subseteq\supp\nu_2$.

If $\supp\tilde\nu_1=\supp\tilde\nu_2$ the result holds taking $\nu_1=\tilde\nu_1$ and $\nu_2=\tilde\nu_2$.

Consider now the second case when $\supp\tilde\nu_1\not=\supp\tilde\nu_2$. Both sets $\supp\tilde\nu_i$ are $\Gamma$--invariant, therefore in view of Lemma \ref{lem:two invariant sets} they must have nonempty intersection, i.e. $K=\supp\tilde\nu_1\cap \supp\tilde\nu_2\neq\emptyset$.
 Since $K$ is $\Gamma$--invariant, by (\Rec) there exists an invariant ergodic Radon measure, say $\nu_1$, whose support is contained in  $K$. Keeping in mind that both sets  $\supp\tilde\nu_1$ and $\supp \tilde\nu_2$ are different, at least one of them, say $\supp \tilde\nu_2$, must be greater than $K$. Then the couple  $\nu_1$ and  $\nu_2:=\tilde\nu_2$ satisfies the claim.

Observe that conditions (\Cont) and (\UB) imply that $M_1:=\supp\nu_1$ is not discrete. Indeed, if $\I$ is the interval appearing in (\Cont), then for all compact intervals $J$
$$\card(M_1\cap J)=\card(g(M_1\cap J))\leq\card(M_1\cap \I),$$
where $g\in \Gamma$ is such that $g(J)\subseteq \I$, by the fact that $M_1$ is $\Gamma$--invariant. Further, since $M_1$ is unbounded one can choose a sequence of compact intervals $J_n$ such that $\card(M_1\cap J_n)\to \infty$. Hence $\card(M_1\cap \I)=\infty$.

Point (1) of Theorem \ref{thm:main2} yields $\nu_1=C\nu_2$, which leads to a contradiction. The proof is complete. \end{proof}

\section{Examples and applications}\label{sect: exp}
We will provide in this section some criteria to ensure recurrence and contraction of the system. In particular, we will focus on the study of systems induced by homeomorphisms for which we can control their behaviour at the end points, such as asymptotically linear homeomorphisms and $C^2$-diffeomorphisms of the interval.

\subsection{Asymptotically linear systems}
In this section we will focus  on the study of systems induced by homeomorphisms that have a linear bound in the sense that for all $g\in \supp \mu$  there exist three positive numbers  $\Aplus (g), \Aminus (g)$ and $B(g)$ such that
\begin{equation}\label{eq: lin bound}
-\Aminus (g)x^--B(g)\leq g(x)\le \Aplus (g)x^++B(g),\qquad\mbox{for all $x\in\mathbb R$,}\end{equation}
where $x^+=\max\{0,x\}$ and $x^-=\max\{0,-x\}$.
\begin{center}
	\includegraphics[width=0.4\linewidth]{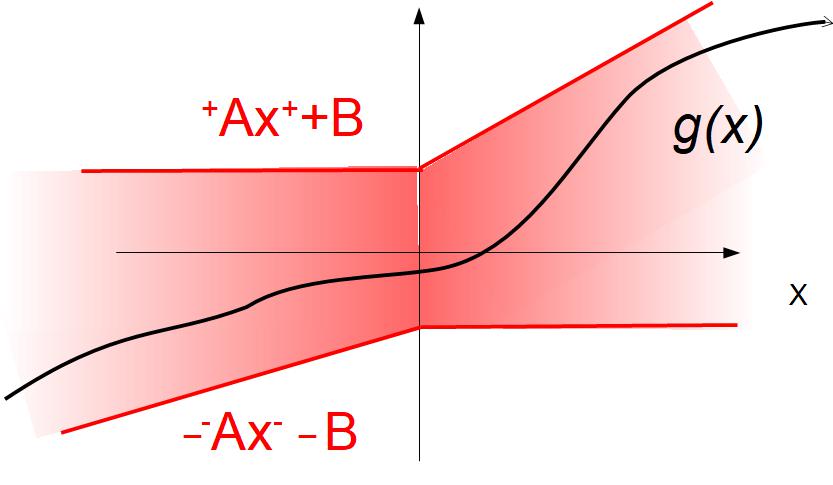}
\end{center}
It can be easily shown that  $g\in\Homr$ satisfies (\ref{eq: lin bound}) iff
the limits
$$\Aplus (g):=\limsup_{x\to +\infty} \frac{g(x)}{x}, \qquad \Aminus (g):=\limsup_{x\to -\infty} \frac{g(x)}{x} \quad \mbox{are finite}$$ and
$$\limsup_{x\to +\infty}[g(x)-\Aplus (g)x]<\infty  \quad\mbox{ and   }\quad \liminf_{x\to -\infty}[g(x)-\Aminus (g)x]>-\infty.$$
This kind of processes appears in many contexts of  probability and related fields and have been investigated in several paper in the last years, see e.g. \cite{Alsmeyer, Alsmeyer:Brofferio:Buraczewski,Brofferio:Buraczewski,Diaconis:Freedman,Goldie}.  A fundamental example that has been widely studied is the  affine recursion where $g(x)=A(g)x+B(g)$ (see \cite{buraczewski2016stochastic} for a general overview).  We refer to \cite[Sect. 6]{Brofferio:Buraczewski} for more detailed presentation of possible applications. In particular, condition (\ref{eq: lin bound}) holds, after conjugation, for  any increasing $C^2$-diffeomorphism $h$ of the interval $[0,1]$, as we will see in the next section.

One expects that if $\Aplus $ and $\Aminus $ are sufficiently often smaller than $1$, then the system will be often repelled  away from infinity and thus will be recurrent or contracting.
For instance one can prove the following sufficient criteria for hypothesis (\Cont) and (\Rec):

\begin{lem}\label{lem:condAL}
\begin{enumerate}
\item 	Suppose that there exists  $g\in \supp\mu $ such that (\ref{eq: lin bound}) holds  with $\Aplus (g)$ and $\Aminus (g)$  smaller than 1. Then (\Cont) holds.
		\item Suppose that (\ref{eq: lin bound}) holds  for every $g\in \supp\mu$. Then (\Rec) holds in any of the following case:
		\begin{enumerate}
			\item $\log\ps{\pm}A(g)$ and $\log^+ B(g)$  are $\mu$-integrable and
			$\int \log\ps{\pm}A(g) \d\mu(g) <0$;
			\item the support of $\mu$ is finite and $\int \log\ps{\pm}A(g  )\d\mu(g) \leq 0$;
			\item $\Aminus=\Aplus=A$,  $\log\ps{\pm}A(g)$ and $\log^+ B(g)$  are $2+\varepsilon$-integrable, for some $\varepsilon>0$, and
			$\int \log A(g) \d\mu(g)=0$.
		\end{enumerate}
	\end{enumerate}

\end{lem}	
\begin{proof}
We are going to use the linear bound assumed in (\ref{eq: lin bound}) to compare the Markov chain $X_n^x=\ell_n(x)=\g_n\cdots \g_1(x)$ with the affine recursions:
\begin{align*}
\Yplus_n^x:&=\Aplus_n \Yplus_{n-1}+B_n,\quad \Yplus_0=x^+,\\
\Yminus_n^x:&=\Aminus_n \Yminus_{n-1}+B_n,\quad \Yminus_0=x^-,
\end{align*}
where $\Aplus_n=\Aplus(\g_n), \Aminus_n=\Aminus(\g_n)$ and $B_n=B(\g_n).$
It can then be verified by the inductive argument that
\begin{equation}\label{eq:ALvsAFF}
-\Yminus_n^x \leq \g_n\cdots \g_1(x)\leq \Yplus_n^x
\end{equation}

\textbf{Proof of (1)} Let $g\in \supp\mu$ be such that $$A:=\max \{\Aplus (g),\Aminus (g)\}<1$$ and set $B:=B(g)$.
It can  be verified by induction (or applying (\ref{eq:ALvsAFF}) when $\g_i=g$ for all $i$'s) that
\begin{equation}
-A^n x^- -\sum_{k=0}^{n-1} A^k B \leq g^n(x)\leq A^n x^+ +\sum_{k=0}^{n-1} A^k B.
\end{equation}
In particular, if $\beta:=\sum_{k=0}^{\infty} $$A^k B$  then $|g^n(x)|\leq A^n|x|+\beta$.
Fix $I:=[-2\beta,2\beta]$ and take any interval $J=[a,b]$. Then for any sufficiently large $n$ we obtain
$$
g^n(J)\subseteq[-A^n|a|-\beta,A^n|b|+\beta]\subseteq I.
$$

\textbf{Proof of (2): (a) and (c).} It is known that under hypotheses (a) or (c) the two dimensional Markov Chain $\{(\Yplus_n,\Yminus_n)\}_{n\in\N}$ is recurrent in $\R^2$, that is, there exists a constant $K>0$ such that for any starting point, with probability 1,
$\max\{|\Yplus_n|,|\Yminus_n|\}<K$ for infinitely many $n$ (see \cite{Babillot:Bougerol:Elie} and Section 4.4.10 in \cite{buraczewski2016stochastic}). From (\ref{eq:ALvsAFF}) it follows that $X_n^x$ visits infinitely often the interval $I=[-K,K]$.

\textbf{Proof of (2): (b).} Under hypothesis (b) one needs to be more careful. In fact, in this case  each of the one dimensional affine recursions $\Yplus_n$ and $\Yminus_n$ is recurrent, but the joint process $(\Yplus_n,\Yminus_n)$ may not.

Let $K>0$ be such that for all $x\in\R$ we have
$$
\P(|\Yplus_n^x|< K \mbox{ i.o.})=1\mbox{ and }\P(|\Yminus_n^x|< K \mbox{ i.o.})=1.
$$
In view of (\ref{eq:ALvsAFF}) this yields that for any $x\in \R$ the set
$$\Omega_x:=\{\ug \in \Gamma^\N|\,\, X_n^x=\g_n\cdots \g_1(x)< K \mbox{ i.o. and } X_n^x=\g_n\cdots \g_1(x)> - K \mbox{ i.o.} \}$$
has full measure.

Observe that since $\Gamma$ is finite,   $K_0:=K \vee \max_{g\in \Gamma} g(-K)<\infty$.
We are going to prove that for any $x\in \R$ and $\ug\in \Omega_x$  the event $[|X_n^x| < K_0]$ occurs, with probability $1$, for infinitely many $n$.

Take  $\ug\in \Omega_x$. Then there are two possible situations :\\
\textbf{ Case 1:   $\g_n\cdots \g_1(x)=X_n^x\leq -K$ infinitely often.} Thus since we also have $X_n^x>-K$ i.o., $X_n^x$ has to cross $-K$ i.o., that is the sequence of stopping times:
$$T_n=T_n(\ug):=\inf\{k>T_{n-1}| X_{k-1}\leq -K \mbox{ and } X_k>-K \}\qquad T_0:=0$$
is almost surely finite for any $n$. Furthermore
$$-K_0\leq -K \leq X_{T_n}^x=\g_{T_n}(X_{T_{n}-1})\leq \g_{T_n}(-K)\leq \max_{g\in \Gamma} g(-K)<K_0$$
thus $|X_n^x|\leq K_0$  infinitely often.\\
\textbf{ Case 2:   $\g_n\cdots \g_1(x)=X_n^x<-K$ only a finite number of times.} Thus $X_n^x>-K$ for all $n\geq N(\ug)$. Then, since $X_n^x<K$ i.o., we have $|X_n^x|< K\leq K_0$ i.o. This completes the proof. \end{proof}

\subsection{$C^2$-diffeomorphisms of the interval}
Our main theorems and the above mentioned results concerning asymptotically linear systems  can be applied to Stochastic  Dynamical Systems on the interval generated by increasing $C^2$-diffeomorphism  of  $[0,1]$.
 Similar iterated function systems have been extensively studied recently (see \cite{alseda_misiurewicz, Czudek_Szarek, Gharaei_Homburg, malicet}).  A sufficient criterion for the uniqueness of an invariant measure in this situation has been stated in Corollary \ref{cor:C2diff} and is a direct consequence of Theorem \ref{thm:main}, Lemma \ref{lem:condAL} and of the following:
 \begin{lem}\label{lem:C2vsAL}
 	Take the diffeomorphism of $(0,1)$ onto  $\R$ defined by $r(u) :=  - \frac 1u + \frac 1{1-u}.$ Then for any increasing $C^2$-diffeomorphism $h$ of  $[0,1]$,
 	 the conjugated homeomorphism  $$ h_r: =  r\circ h \circ r^{-1}\in\Homr  $$ satisfies (\ref{eq: lin bound}) with
 	$$\Aplus (h_r)=\frac{1}{h'(1)} \quad\mbox{ and   }\quad \Aminus (h_r)=\frac{1}{h'(0)}.$$
 Furthermore, if $\mu$ is  a finitely  supported measure on the family of increasing diffeomorphisms in $C^2([0,1])$ and $\mu_r$ is the conjugated measure on $\Hom^+(\R)$, a Radon measure $\nu$ on $(0,1)$ is $\mu$--invariant iff the Radon measure on $\R$ of the form $$\nu_r(f)=r*\nu(f)=\int_{[0,1]} f(r(x))d\nu(x)
 	$$
 	is $\mu_r$--invariant.
 \end{lem}

 \begin{proof}
We have
 \begin{align*}
 \Aplus (h_r) &=\limsup_{x\to +\infty} \frac{h_r(x)}{x}
 =\limsup_{x\to +\infty} \frac{r(h(r^{-1}(x))}{r(r^{-1}(x))}\\
 &=\limsup_{u\to 1^-} \frac{r(h(u))}{r(u)}\quad\mbox{ change of variable }u:=r^{-1}(x)\\
 &=\limsup_{u\to 1^-} \frac{\frac {1}{1-h(u)}}{\frac 1{1-u}}\quad\mbox{ since  $r(u)\sim \frac 1{1-u}$ for $u\sim 1^-$}\\
 &=\limsup_{u\to 1^-} \frac{1-u}{h(1)-h(u)}=\frac{1}{h'(1)}\quad\mbox{ since  $h(1)=1$.}
 \end{align*}
 Furthermore,  since $h$ is $C^2(0,1)$, we have
 $h(u)=1+h'(1)(1-u)+\mathcal{O}((1-u)^2)$. Thus finally
 \begin{align*}
 \limsup_{x\to+ \infty } \left[h_r(x)-\frac{x}{h'(1)}\right]&=\limsup_{u\to 1^-}\left[\frac{1}{1-h(u)}-\frac{1}{h'(1)(1-u)}\right]\\&=\limsup_{u\to 1^-}\frac{\mathcal{O}((1-u)^2)}{h'(1)^2(1-u)^2}<\infty.
 \end{align*}
 Similar calculations can be done near $-\infty$ and $0$.

 The second part of the lemma is obvious.
 \end{proof}

\subsection{Counterexamples to Theorems \ref{thm:main} and \ref{thm:main2} }
In this section we intend to provide some examples of Stochastic Dynamical Systems that have more than one invariant measure to explain that neither condition (\Cont) nor (\UB) is sufficient alone to guarantee  uniqueness.

\subsubsection{Contraction but not unboundedness}\label{ssect:CnotUB}
Consider the Stochastic Dynamical Systems generated by a set $\Gamma$ of homeomorphisms that fix two distinct points $a$ and $b$ of $\R$ and are all repulsive at the end point. For instance take $\mu$ that gives mass 1/2 to $h(x)=x^{1/3}$ and to  $k(x)=x^{1/5}$.
Then $\mu$ is contracting because $h^n[-K,K]=[-K^{1/3^n},K^{1/3^n}]$ is in $[-2,2]$ for any large $n$. The interval $[-2,2]$  is also recurrent for similar reasons.  This system does not have a unique invariant measure since $\delta_0$ and $\delta_1$, the Dirac measures in $0$ and $1$, are both $\mu$--invariant.

\subsubsection{Unboundedness but not contraction }\label{ssect:UBnotC}
An example of a recurrent Stochastic Dynamical System  that satisfies (\UB) but not (\Cont) is just given by the simple random walk on $\Z\subset\R$. In fact, take $\mu$ that gives mass 1/3 to $h_0(x)=x$, $h_+(x)=x+1$ and  $h_-(x)=x-1\in \Homr$. It defines a recurrent Markov chain and it is obviously unbounded. It possesses infinitely many Radon ergodic invariant measures given by the counting measures on $\Z+k\subseteq \R$ for any $k\in[0,1)$. The  Lebesgue measure on $\R$ is also invariant but it is not ergodic.

\subsubsection{Ergodic measures with non minimal support}\label{ssect:ErgNonMin}
We propose here an example to prove that an ergodic measure may have support that is not minimal. The idea is to start with a Stochastic Dynamical System generated by a measure  $\overline\mu$ on the set of  increasing $C^2$-diffeomorphisms  of  $[0,1]$  that has a unique Radon measure $\overline{\nu}$   whose support is the whole interval $(0,1)$. It follows then that  $\overline{\nu}$ is ergodic. Let $\overline\Gamma=\supp\overline\mu$. For any $\overline g \in \overline\Gamma$ define three homeomorphisms of $\R$:
$$
g_0(x):=\overline{g}(\{x\})+\lfloor x\rfloor,\quad
g_+(x):=\overline{g}(\{x\})+\lfloor x\rfloor +1, \quad
g_-(x):=\overline{g}(\{x\})+\lfloor x\rfloor -1,
$$
where $\{x\}$ is the fractional part of $x$  and $\lfloor x\rfloor$ the floor function. Heuristically the function $g_0$ fixes each integer interval $[n,n+1]$ and acts on each one of them as $\overline{g}$, while $g_\pm$ do the same but are then composed with a translation by $\pm 1$.
Let $\mu$ be the measure charging $g_0,g_\pm$ with mass equal to $\overline{\mu}(\overline{g})/3$. Then it can be proved that the measure
$$\nu(f):=\sum_{k=-\infty}^{+\infty}\int_0^1f(y+k)d\overline{\nu}(y)$$
is a $\mu$-invariant Radon ergodic measure whose support is the whole $\R$. On the other hand, $\Z\subset\R$ is a discrete closed invariant set for $\mu$ (and the counting measure on $\Z$ is another ergodic measure).

\subsubsection{Non recurrent system} A classical example that shows that for non recurrent systems a closed minimal $\Gamma$--invariant set can be a support of several invariant measures is a non centred random walk on $\Z$. Suppose that $g(x)=x+B(g)$ with $B(g)\in\Z$.   Further, suppose also that  $\E(B(\g_1))\not= 0$ and that there exists $\alpha\not=0$ such that $\E(e^{-\alpha B(\g_1)})=1$. Then both the counting measure on $\Z$ and  the measure on $\Z$ such that
$\nu(x)=e^{\alpha x}$ for any  $x\in\Z$ are invariant.

\subsubsection{Non-Radon invariant measures} The restriction to Radon measures in Theorem \ref{thm:main} is indispensable. In the family of  Borel measures the uniqueness
of the invariant measure can be easily broken, see e.g. Remark 2 in \cite{Babillot:Bougerol:Elie}.

\section{Appendix: Some results on ergodic invariant measures  for Markov--Feller processes  }\label{sec:ergodic}
This part of our paper is
devoted to the description of ergodic measures and to the proof of an ergodic decomposition for Markov--Feller processes on locally compact metric spaces \eqref{e1_20.02.13}.
Some of the results of this section seem to be classical and have been often used in a different context in several works in this fields. They are based on the classical theory of positive contractions of $L^1$-spaces that is a powerful and general tool. However we could not find a comprehensive reference specifically adapted to study of Markov--Feller processes with an invariant Radon measure.  So we  give a quick survey of the results that we need in our paper and explain how they can be deduced from the general theory.
In particular,  we give an explicit proof of the ergodic decomposition  of a general invariant Radon measure as an integral over the class of ergodic Radon measures.

For a complete overview on the ergodic theory and related infinite measures, $L^1$-contractions and Markov processes we refer to the books by S. R. Foguel \cite{Foguel}, A. M. Garsia~\cite{Garsia} and D. Revuz~\cite{Revuz}. For a glimpse to the theory we also suggest the nice informal survey of R. Zweim\"uller \cite{zweimueller}.

\subsection{ Markov--Feller processes  and $L^1$-contractions}
Let $(X, \rho)$ be a locally compact metric space and let $(X,  \mathcal B(X), \nu)$ be a $\sigma$--finite measure space with  Radon measure $\nu$. Let $P$ be an operator on $L^1(X,\nu)$ and $L^\infty(X,\nu)$ such that
\begin{itemize}
	\item $P$  is \textit{positive}, i.e. $Pf\ge 0$  for $f\ge 0$;
	\item $P{\bf 1}_X(x)={\bf 1}_X(x)$ for $\nu$-almost every $x\in X$;
	\item $P$ is a {\it contraction} of $L^1(X,\nu)$, i.e. the operator norm $\|P\|_1$ on this space is less than $1$.  This last condition is equivalent to the property of the measure $\nu$ called an \textit{excessive measure}, that is $$\int_X Pf(x)d\nu(x)\leq \int_X f(x)d\nu(x)\quad\mbox{for}\quad f\in L^1(X,\nu), f\ge 0.$$
\end{itemize}	
   We shall call the quadruple $(X, \mathcal B(X), \nu, P)$ a Markov process. This process is said to be \textit{Feller} if  additionally we assume that $Pf\in C(X)$
 for $f\in C_c(X)$. Here $C(X)$ denotes the space of continuous functions and  $C_c(X)$  the sub-space of continuous functions with compact support.
A Radon measure $\nu$
 will be called {\it invariant} for a given Markov--Feller process if $ \nu(Pf)= \nu(f)$ for $f\in L^1(X,\nu)$.

Using the duality between $L^1(X,\nu)$ and $L^\infty(X,\nu)$ we can define a dual operator $P^*$ on both  $L^1(X,\nu)$ and $L^\infty(X,\nu)$. More precisely, for any $f\in L^1(X,\nu)$ (respectively $f\in L^\infty(X,\nu)$) $P^*f$ is the unique function in $L^1(X,\nu)$ (respectively in $L^\infty(X,\nu)$) such that
\begin{equation*}
\int_X P^*f(x) g(x)d\nu(x) = \int_X f(x) Pg(x)d\nu(x)\quad\mbox{ for all $g\in L^\infty(X,\nu)$ (resp. $g\in L^1(X,\nu)$  )}
\end{equation*}
It can be easily checked that also $P^*$ is a positive contraction of $L^1(X,\nu)$.

We can associate to a Markov operator $P$ with an invariant measure $\nu$ the  \textit{space of the trajectories} of the associated Markov chain $(X_n)_{n\in\N}$, that is, the product space $X^\N$ equipped with the measure $\P_\nu$  such that for any finite collection of compact sets  $I_i\subset X $, $i=0,\ldots, n$, the measure of the cylinder $[I]=I_0\times\cdots \times I_n\times \R\times \cdots$ is given by the formula:
\begin{align*}
\P_\nu([I])&= \P_\nu(X_0\in I_0,\ldots, X_n\in I_n )\\&:=\int_{\R^{n+1}}  {\bf 1}_{I_n}(x_n)\cdots {\bf 1}_{I_0}(x_0)P(x_{n-1}, dx_{n}) \cdots P(x_1, dx_2)P(x_0, dx_1) \nu(dx_0).
\end{align*}
The shift $\tau$ on $X^\N$, that is the map $\ux=(x_0,x_1,\ldots)\mapsto \tau\ux=(x_1,x_2,\ldots)$, induces the operator on $L^1(X^\N,\P_\nu)$ (and on $L^\infty(X^\N,\P_\nu)$) that will also be denoted by $\tau$ and defined by the formula: $\tau f(\ux)=f(\tau\ux)$.
If $\nu$ is $P$-invariant, then $\P_\nu$ is $\tau$-invariant. Thus $\tau$ is a positive contraction of $L^1(X^\N,\P_\nu)$.

More generally,  let $(W,\omega)$ be a $\sigma$-finite measure space. A  linear operator $T$ of  $L^1(W,\omega)$ is a \textit{positive contraction} if it is positive ($Tf\geq 0$ whenever $f\geq 0$) and   $\|T\|_1\leq 1$.
We may define the adjoint operator $T^*: L^{\infty}( \omega)\to L^{\infty}( \omega)$ by the formula:
$$
\int_W T^*g f\d \nu=\int_X g Tf\d \nu\qquad \mbox{for $g\in L^{\infty}( \nu)$ and $f\in L^1( \nu)$.}
$$

As we have seen above, we can associate to a Markov kernel $P$ with a $P$-invariant measure $\nu$  at least three different contractions: the  contraction $P$ on $L^1(X,\nu)$, the contraction $P^*$ also on $L^1(X,\nu)$ and the shift $\tau$ on $L^1(X^\N,\P_\nu)$. All of them  possess interesting properties, but this abundance generates also some confusion. We will be mainly interested  in the contractions $P$ and $\tau$. Let us just notice that the contraction $P^*$  is particularly adapted and widely used to the study of Harris recurrent Markov chains (see for instance  Revuz \cite{Revuz} and Foguel \cite{Foguel}), but this is not our case.

These three contractions are deeply related and the dynamical systems they engender share often the same ergodic properties, as it will be shown below.  For some results in this direction see also the recent paper of F. P\`ene and D. Thomine \cite[Section 2]{Pene:Thomine}.

\subsection{Ergodic measures}\label{sec: appendix erg mes}
A fundamental property of $L^1$--contractions is ergodicity saying that the space cannot be decomposed into smaller invariant pieces. More precisely, Borel set $A$ is called {\it $T$--invariant} (or \textit{invariant} ) if $T^*{\bf 1}_A={\bf 1}_A$ (or equivalently if $\nu_A$, the restriction of $\nu$ to $A$, is a $T$--invariant measure). An invariant measure $\nu$ is called {\it ergodic} if either $\nu(A)=0$ or $ \nu(X\setminus A)=0$ for any $T$--invariant set $A\subset X$.

For the contractions induced by a Markov operator $P$, we have in principle at least three definitions of ergodicity (for $P$, $P^*$ and $\tau$) but all of them coincide. First of all, observe that the $\sigma$-algebras of  invariant sets defined by the contractions $P$ and $P^*$ on $L^1(x,\nu)$ coincide (see \cite[Chapter 4, Proposition 3.4]{Revuz}).  Thus $\nu$ is $P$--ergodic if and only if it is $P^*$--ergodic. Furthermore, we have the following:

\begin{lem}\label{lem: erg nu Pnu}
	Let $P$ be a Markov--Feller operator, and let $\nu$ be an invariant measure.
	If $\nu$ is $P$-ergodic, then $\P_\nu$ is ergodic for the shift $\tau$.
\end{lem}
\begin{proof}
	Let $\uA\subset X^\N$ be $\tau$-invariant, i.e. $\tau^{-1}\uA=\uA$ up to $\P_\nu$-null measure set. Let
	$$u(x)=\P_x(\uA)=\P_x((X_n)_{n\geq 0}\in \uA).
$$
	Observe that for $\nu$-a.e. $x$
	$$u(x)=\P_x(\tau^{-1}\uA)= \P_x((X_{n+1})_{n\geq 0}\in \uA)=\E_x(\P_{X_1}(\uA))=Pu(x).$$
	Thus $u$ is $P$-invariant and  since $\nu$ is $P^*$-ergodic,  $u(x)=u_0$ for some constant $u_0$.
	
	Take now $B_n$ in $\sigma(X_0, X_1,\ldots, X_n)$ -- the $\sigma$-algebra generated by the first $n+1$ coordinates. Then
	$$\P_\nu(B_n\cap \uA)=\P_\nu(B_n\cap \tau^{-(n+1)}\uA)=\E_\nu({\bf 1}_{B_n}\P_{X_{n+1}}(\uA))=u_0\P_\nu(B_n).$$
	Since the set of functions ${\bf 1}_{B_n}$ spans a dense subset of  $L^1(\P_\nu)$, we see that ${\bf 1}_{\uA}=u_0$ has to be constant, i.e. $\uA$ is $\P_\nu$ trivial. This completes the proof.
\end{proof}

In the specific context of this paper where $P$ is induced by the action of the discrete measure $\mu$ on the group of $\Hom(X)$, we can characterize invariant sets (and prove directly that $P$ and $P^*$--invariant sets coincide).
\begin{lem}\label{lem:erg-car}
			Let $P$ be the Markov operator  defined by the action of a discrete distribution $\mu$ on the group of homeomorphims of $X$  as in (\ref{e2_20.02.13}), and let $\nu$ be an invariant Radon measure. Then for any measurable set $A\subset X$ the following conditions are equivalent:
	\begin{enumerate}
		\item $\nu(A\bigtriangleup g^{-1}A)=0$ for each $g\in \Gamma$;
		\item $P^*{\bf 1}_A={\bf 1}_A$ in $L^\infty(X,\nu)$;
		\item $P{\bf 1}_A={\bf 1}_A$ in $L^\infty(X,\nu)$.
	\end{enumerate}
In particular, if  $M$ is a closed $\Gamma$-invariant set, then M is $P$-invariant for  the Markov operator~$P$.
\end{lem}
\begin{proof}
	(1)$\Rightarrow$ (2) \\
	Suppose (1) holds. Since $\nu $ is invariant, for any $f\in L^1(\nu)$ we have
	\begin{align*}
	\int_Xf(x) P^*1_A(x)d\nu(x)=&=\int_\Gamma\int_Xf(g(x)){\bf 1}_A(x)d\nu(x)d\mu(g)\\&=\int_\Gamma\int_Xf(g(x)){\bf 1}_A(g(x))d\nu(x)d\mu(g)\\
	&=\int_Xf(x){\bf 1}_A(x)d\nu(x).
	\end{align*}	
	
	(2)$\Rightarrow$ (1) \\
Let $K_n\nearrow X$ be a sequence of increasing compact sets. Let $A^c=X\setminus A$ and $B_n:=K_n\cap A^c$.  Then, since ${\bf 1}_{B_n}\in L^1(\nu)$,
	\begin{align*}
	\sum_{g\in\Gamma}\mu(g)\nu(g^{-1}B_n\cap A)&=\int_X\int_\Gamma{\bf 1}_{B_n}(g(x)){\bf 1}_A(x)d\nu(x)d\mu(g)\\
	&=\int_X P{\bf 1}_{B_n}(x){\bf 1}_A(x)d\nu(x)\\
	&=\int_X {\bf 1}_{B_n}(x)P^*{\bf 1}_A(x)d\nu(x)\\&= \int_X {\bf 1}_{B_n}(x){\bf 1}_A(x)d\nu(x)=\nu(B_n\cap A)=0.
	\end{align*}
	Thus $\nu(g^{-1}B_n\cap A)=0$ for all $g\in \Gamma$ and
	$$\nu((g^{-1}A)^c\cap A)=\nu(g^{-1}(A^c)\cap A)=\lim_{n\to\infty}\nu(g^{-1}B_n\cap A)=0.$$
	Observe that, since $\nu$ is $P$-invariant, also $P^*{\bf 1}_{A^c}={\bf 1}_{A^c}$.
	Similarly 	$\nu(g^{-1}A\cap A^c)=0$ and finally we can conclude that $$	\nu(A\bigtriangleup g^{-1}A)=\nu(g^{-1}A\cap A^c)+\nu((g^{-1}A)^c\cap A)=0.$$
	
	(1)$\Leftrightarrow$ (3)
	Observe that $P{\bf 1}_A(x)=\sum_{g\in\Gamma}\mu(g){\bf 1}_{g^{-1}A}(x)$ thus
	$$P{\bf 1}_A={\bf 1}_A\Longleftrightarrow {\bf 1}_{g^{-1}A}={\bf 1}_A\qquad\forall g\in \Gamma$$ since $\mu(g)>0$ for all $g\in \Gamma$.

		Let $M$ be a closed $\Gamma$-invariant set. Since $M\subseteq   g^{-1}M$, we have then
	$$
	\sum_{g\in\Gamma}\mu(g)\nu(g^{-1}M\bigtriangleup M)=\sum_{g\in\Gamma}\mu(g)(\nu(g^{-1}M)-\nu(M))=0,
	$$
	by the fact that $\nu$ is invariant. Therefore, $\nu(g^{-1}M\bigtriangleup M)=0$ for all $g\in\Gamma$.
\end{proof}


\subsection{Chacon-Ornstein Ergodic Theorem for $L^1$-contractions}\label{sec: appendix CO}

The Chacon--Ornstein Ratio Ergodic Theorem is an extremely powerful and  general theorem to study the asymptotic behaviour of  the partial sums
$$S_nf:=\sum_{k=0}^nT^kf\quad \mbox{with } f\in L^1(W, \omega).$$

\begin{thm}[The Chacon--Ornstein Ergodic Theorem]
Let $T$ be a positive contraction of $L^1(W,\omega)$. Assume that the operator $T$ is \textbf{conservative}, that is, there  exists a strictly positive function $\varPhi\in  L^1(W, \omega)$ such that $\lim_{n\to\infty} S_n \varPhi(w)=+\infty$ for $\omega$-almost all $w\in W$. Then  for any $f\in L^1(\omega)$ the limit
\begin{equation}\label{eq: CO Thm conv}
Lf(w):=\lim_{n\to\infty}\frac{S_nf(w)}{S_n\varPhi(w)}\qquad \mbox{ exists and is finite for  $\omega$-a. e. $w$.}
\end{equation}
Furthermore, the function $Lf$ is invariant (i.e. measurable with respect to $\II$, the $\sigma$-algebra of all $T$--invariant sets) and
\begin{equation}\label{eq: CO Thm int}
\int Lf(x) \Phi(x)\d\omega(x)=\int f(x)\d\omega(x).
\end{equation}
\end{thm}
For a complete proof see for instance \cite[Theorem 2.6.1]{Garsia}.
The Ratio Ergodic Theorem enables us to give  another characterisation of ergodic measures:
\begin{lem}\label{lem: eerg-const lim} Let $\FF$ be a dense family in $L^1(\omega)$. An invariant measure $\omega$ is ergodic iff $Lf$ is constant for all $f\in\FF$. \end{lem}
\begin{proof}
	If $\omega$ is ergodic then the invariant $\sigma$-algebra is trivial and thus $Lf$ is constant. In consequence, by (\ref{eq: CO Thm int}), it is equal to $\frac{\omega(f)}{\omega(\varPhi)}$.

	Suppose now that $Lf=\frac{\omega(f)}{\omega(\varPhi)}$ is $\omega$--a.e. constant
	for any $f\in \FF$.
	Let $A$ be an invariant set. Since
	$T({\bf 1}_A f)={\bf 1}_AT f$ (see for instance \cite[Prop 2.5.6]{Garsia}), then $L({\bf 1}_A f)={\bf 1}_ALf$ $\omega$-a.e. and
	$$
	\omega({\bf 1}_A f)=\omega(\varPhi) L({\bf 1}_A f)(x)= \omega(\varPhi) {\bf 1}_A(x) Lf(x)={\bf 1}_A(x)\omega(f).
	$$ Since $f\in\FF$ is arbitrary and is  $\FF$ a dense family in $L^1(\omega)$, the set  $A$ must be trivial and we are done.	
\end{proof}

A direct consequence of the last lemma and of Lemma \ref{lem: erg nu Pnu}, in the case of Markov-Feller operator, is the following corollary that summarises some of the fundamental results needed  in our paper.
\begin{cor}\label{cor: ergodic shift} 	Let $\nu$ be an ergodic invariant Radon measure for the Markov-Feller operator $P$. Suppose that the Markov chain is recurrent, i.e. there exists a compact set $K$ such that   ${{\bf 1}_K (x_n)+\cdots +{\bf 1}_K(x_0)}\to +\infty$ $\P_\nu$-a.e. Then
	 for any nonnegative function $\phi\in L^1(X,\nu)$ we have $\nu(\phi)>0$ if and only if
	 ${\phi(_n)+\cdots +\phi(x_0)}\to +\infty$ $\P_\nu$-a.e.,
	  and in this case for all $f\in L^1(X,\nu)$:
	$$\lim_{n\to\infty}\frac{f(x_n)+\cdots +f(x_0)}{\phi(x_n)+\cdots +\phi(x_0)}=\frac{\nu(f)}{\nu(\phi)} \quad\P_\nu-\mbox{a.e.}$$
\end{cor}
\begin{proof} Since $\P_\nu$ is $\tau$-ergodic by Lemma \ref{lem: erg nu Pnu}, applying Lemma \ref{lem: eerg-const lim} to $\underline{f}(\ux):=f(x_0)$ and $\underline{\Phi}(\ux):=\Phi(x_0)$ with $\Phi>{\bf 1}_K$, we obtain
	$$\lim_{n\to\infty}\frac{f(x_n)+\cdots +f(x_0)}{\Phi(x_n)+\cdots +\Phi(x_0)}=\lim_{n\to\infty}\frac{\underline{f}(\tau^n\ux)+\cdots +\underline{f}(\ux)}{\underline{\Phi}(\tau^n\ux)+\cdots +\underline{\Phi}(\ux)}=\frac{\P_\nu(\underline{f})}{\P_\nu(\underline{\Phi})}=\frac{\nu(f)}{\nu(\varPhi)}\quad 	\P_\nu-\mbox{a.e.} $$
Take a nonnegative function $\phi\in L^1(X,\nu)$ such that $\nu(\phi)>0$,
	  then $$S_n\phi(x)\sim\frac{\nu(\phi)}{ \nu(\Phi)}S_n\Phi(x)\to\infty.$$
	Conversely, if $\phi(x_n)+\cdots + \phi(x_0)$ tends to $\infty$, $\nu$-a.e., then, since the limit of $S_n\Phi/S_n\phi$ is finite  (see for instance \cite[Chapter III, Theorem D]{Foguel}) and equals to $\nu(\Phi)/\nu(\phi)$, by the previous step we obtain that $\nu(\phi)>0$.
\end{proof}

\subsection{Ergodic Decomposition of invariant measure}

This part of the paper is
devoted to complete proof of an ergodic decomposition for Markov--Feller processes on locally compact metric spaces. From this decomposition formula (\ref{e1_20.02.13}) will follow.
\begin{thm}\label{Thm: erg dec}
Let $(X, \mathcal B(X), \nu, P)$ be a Markov--Feller process. Assume that there exists a function $\varPhi\in C(X)\cap L^1(\nu)$, $\varPhi>0$, such that $\sum_{n=1}^{\infty} P^n\varPhi(x)=+\infty$ for all $x\in X$. Then there exists a measurable set  $X_0\subset X$ with $\nu(X\setminus X_0)=0$ such that:
\begin{itemize}
\item[1)] for every $x\in X_0$ there exists a Radon measure $\nu_x$ such that
\begin{equation}\label{eq: conv nux cont}
\nu_x(f)=\lim_{n\to\infty}\frac{S_nf(x)}{S_n\varPhi(x)}\qquad \mbox{for all $f\in C_c(X)$;}
\end{equation}
\item[2)] for every nonnegative $f\in L^1(\nu)$:
\begin{equation}\label{eq: conv nux mes}
\nu_x(f)=\lim_{n\to\infty}\frac{S_nf(x)}{S_n\varPhi(x)}\qquad \mbox{for $\nu$-a.e. $x\in X$;}
\end{equation}
thus the function $x\mapsto\nu_x(f)$ is measurable and  \begin{equation}\label{eq: erg dec}
\nu(f)=\int_X \nu_x(f) \varPhi(x)\nu(\d x).
\end{equation}
\item[3)] $\nu_x$  is invariant and ergodic for any $x\in X_0$.
\end{itemize}
\end{thm}
 Although the above  result has been used by several authors and  belong to the folklore of the field, we are not aware of any explicit reference in the literature.  In our understanding, in the specific case of  Radon measures invariant under the action of a countable group,  the ergodic decomposition could be deduced (with some work) from the paper of Greschonig and Schmidt \cite[Theorem 1.4]{Greschonig:Schmidt}. However, since  their approach does not seem to apply to more general Markov--Feller processes,  we present here  an independent proof. This can be interesting  in view of the future development of  Stochastic Dynamical Systems induced by transformation $g_i$ that are not invertible or  not countably generated.\\
We would also like to mention that in  the ergodic decomposition obtained in the previous theorem, the set of ergodic measures $\nu_x$ \textit{depends} on the measure $\nu$. In this sense our result is weaker than the one proved in \cite{Greschonig:Schmidt}, where the authors  acquire the existence of the set of quasi-invariant ergodic measures that depends only on the group action.

\medskip

\begin{proof}

First observe that since $X$ is a locally compact metric space, there exist a countable increasing family of compact sets $(K _i)_{i\in\N}$ such that $K _i\nearrow X$ and a countable family of continuous functions $\mathcal F\subset C_C(X)$ dense in    the space $C_c(X)$ (with the supremum norm) and such that 
if the support of $f$ is contained in $K _i$, then for every $\varepsilon>0$ there exists $h\in\FF$ such that
\begin{equation}\label{eq: family FF1}
\|f-h\|_\infty<\varepsilon\mbox{ and  } \supp h\subset K_ {i+1}.
\end{equation}
Thus,  for every $f\in C_c(X)$ and $\delta>0$, there exists $h\in \FF$ such that
$$ |f(x)-h(x)|<\delta\ \Phi(x) \qquad \mbox{ for all } x\in X.$$
Indeed, since $C_{i+1}=\inf_{x\in K_{i+1}} \Phi(x)>0$, it suffices  to take $h\in \FF$ such that (\ref{eq: family FF1}) holds for $\varepsilon=\delta/C_{i+1}$.

We will split the proof into four steps.
\vskip3mm
{\bf Step I.} 
\textit{We are going to define measures $\nu_x$ for $\nu$-almost all $x\in X$ and to prove 1). } \\
Let $  X_1$ be the set of all $x\in X$ such that
$$
Lh(x):=\lim_{n\to\infty}\frac{S_nh(x)}{S_n\varPhi(x)}\qquad \mbox{exists for all $h\in\mathcal F$.}
$$
Since  $\FF\subset L^1(\nu)$ is countable, by the Chacon-Orstein Theorem, $\nu (X\setminus X_1)=0$.

 We shall prove that if $x\in  X_1$, then the above limit exists for an arbitrary $f\in C_c(X)$. For this purpose  we check that the sequence
$(\frac{S_nf(x)}{S_n\varPhi(x)})_{n\in\N}$ for $f\in C_c(X)$ and $x\in  X_1$ satisfies the Cauchy condition.
Fix $f\in C_c(X)$ and $\varepsilon>0$. Let $h\in\mathcal F$ be such that $ \|f-h\|_{\infty}<(\varepsilon/3) \ \Phi.$
Then, we have
$$
\begin{aligned}
\left|\frac{S_nf(x)}{S_n\varPhi(x)}-\frac{S_mf(x)}{S_m\varPhi(x)}\right|&\le
\frac{S_n|h-f|(x)}{S_n\varPhi(x)}+ \frac{S_m|h-f|(x)}{S_m\varPhi(x)}
+\left|\frac{S_nh(x)}{S_n\varPhi(x)}-\frac{S_mh(x)}{S_m\varPhi(x)}\right|\\
&\le \frac{\varepsilon}{3}+\frac{\varepsilon}{3}+\frac{\varepsilon}{3}=\varepsilon,
\end{aligned}
$$
for all $m, n\in\mathbb N$ sufficiently large. Since $\varepsilon>0$ was arbitrary, the Cauchy condition is verified.

Define now, for any $x\in  X_1$, the functional on $C_c(X)$ by the formula:
$$
f\mapsto Lf(x)=\lim_{n\to\infty}\frac{S_nf(x)}{S_n\varPhi(x)}\qquad \mbox{for $f\in C_c(X)$}.
$$
Since this is a positive linear functional, it is represented by some regular measure $\nu_x$, i.e. $Lf(x)=\nu_x(f)$ for  $f\in C_c(X)$, by the Riesz-Markov-Kakutani 
Representation Theorem. Obviously, $\nu_x$ is a Radon measure. This proves (\ref{eq: conv nux cont}) for all $x\in X_0\subseteq X_1$ of full measure.

\vskip3mm
{\bf Step II.} \textit{ We shall check that for any $f\in L^{1}(\nu)$ we have:}
\textit{\begin{equation}\label{eq: nux vs Lx}
	\nu_x(f)=Lf(x)\quad\mbox{for $\nu$--almost all $x\in X$} \end{equation}
and prove 2).}	
By (\ref{eq: conv nux cont}), we already know that the latter equality is true for all $f\in C_c(X)$. We are going to prove that by a continuity argument it can be extended to all functions $f\in L^1(\nu)$. However, observe that if the function is not continuous,  the set of $x$ where (\ref{eq: nux vs Lx}) holds may depend on $f$.	

Let $g_n$, $n\in\N$, be a nonincreasing family of nonnegative measurable functions such that $g_n\searrow 0$ in $L^1(\nu)$. Then $Lg_n(x)\searrow 0$ for $\nu$-a.e $x\in X$. In  fact, since the   operator $L$ is  positive and  $Lg_n$ is a nonincreasing sequence of measurable functions,  the limit
$\overline{g}(x):=\lim_{n\to\infty}Lg_n(x)$
exists for $\nu$-a.e. $x$ and is nonnegative. Furthermore,
\begin{align*}
\int_X \overline{g}(x)\Phi(x)\nu(\d x)
&=\int_X\lim_{n\to\infty}Lg_n(x)\varPhi(x)\nu(\d x)\\
&\le \liminf_{n\to\infty}\int_X Lg_n(x)\varPhi(x)\nu(\d x) \quad\mbox{by Fatou's Lemma}\\
&=  \liminf_{n\to\infty}\int_X g_n(x)\nu(\d x) \quad\mbox{by Chacon-Orstein Theorem}\\
&=0.
\end{align*}
Thus $0=\overline{g}(x):=\lim_{n\to\infty}Lg_n(x)$ for $\nu$-a.e $x$.

Let consider the class of function:
 $$\HH:=\left\{f\mbox{ bounded mesurable function on $X$}\ |\ \nu_x(f\varPhi)=L(f\varPhi)(x)\quad \mbox{$\nu$-a.s.}\right\}.$$
Observe that: \begin{itemize}
	\item \textit{If $f_n\in \HH$  is a family of nonnegative and increasing bounded functions converging to $f$,
	then $f\in\HH$.} In fact, since $f_n\varPhi\nearrow f\varPhi$ pointwise and in $L^1(\nu)$:
\begin{align*}
\nu_x(f\varPhi)&=\lim_{n\to\infty}\nu_x(f_n\varPhi) \quad\mbox{by the Monotone Convergence Theorem}\\
&=\lim_{n\to\infty}L(f_n\varPhi)(x)\quad\mbox{since $f_n\in\HH$}\\
&=\lim_{n\to\infty}[L(f\varPhi)(x)-L((f-f_n)\varPhi)(x)] \quad\mbox{by linearity of $L$}\\
&=L(f\varPhi)(x)-\lim_{n\to\infty}L((f-f_n)\varPhi)(x)=L(f\varPhi)(x) \quad\mbox{}
\end{align*}
since  $g_n=(f-f_n)\varPhi\searrow 0 $ pointwise and in $L^1(\nu)$, by the Dominated Convergence Theorem.
\item \textit{If $U$ is an open subset of $X$, then ${\bf 1}_U\in\HH$.} In fact, there exists a nondecreasing sequence of nonnegative functions $f_n\in C_c(X)$ such that $f_n\nearrow {\bf 1}_U$ pointwise.  Since $f_n\varPhi\in C_c(x)$, step I yields $f_n\in \HH$, and consequently ${\bf 1}_U\in \HH$.
\item \textit{If $f, g \in \HH$, then $f + g$ and $cf$ are in $\HH$ for any real number c.} This is a direct consequence of linearity of $\nu_x$ and $L$.
\end{itemize}
Applying the Monotone Class Theorem for functions (see for  instance \cite[Thm. 5.2.2]{Durrett}), $\HH$ contains all measurable bounded functions.

Take now a nonnegative $f\in L^1(\nu)$ and an increasing sequence of compact sets $K_n\nearrow X$. Observe that
$$f_n(x):=\frac{f(x)\wedge n}{\varPhi(x)} \quad\mbox{for $x\in K_n,$}$$
and $f_n(x)=0$ otherwise. It is  easy to check that $f_n$ are bounded and $f_n\varPhi\nearrow f$, both pointwise and in $L^1(\nu)$, thus, following the same reasoning as above, we obtain
\begin{align*}
\nu_x(f)&=\lim_{n\to\infty}\nu_x(f_n\varPhi) \quad\mbox{by the Monotone Convergence Theorem}\\
&=\lim_{n\to\infty}L(f_n\varPhi)(x)\quad\mbox{since $f_n\in\HH$}\\
&=\lim_{n\to\infty}L(f)(x)-L(f-f_n\varPhi)(x) \quad\mbox{by linearity of $L$}\\
&=Lf(x),\end{align*}
since $g_n=f-f_n\varPhi\searrow 0 $ pointwise and in $L^1(\nu).$ This completes, invoking (\ref{eq: CO Thm int}), the proof of 2).

\textbf{Step III. } \textit{We are going to prove that there exists a set of full measure $X_2$ such that $\nu_x$ is $P$-invariant for all $x\in X_2$. }
Let $X_2$ be the set of all $x\in X_1$ such that:
\begin{enumerate}
	\item $\nu_x(f)=Lf(x)$ for all $f\in\FF$;
	\item $\nu_x(Pf)=L(Pf)(x)$ for all $f\in\FF$;
	\item $\nu_x(\varPhi)=L\varPhi(x)$ and $\nu_x(P\varPhi)=L(P\varPhi)(x)$.
\end{enumerate}
Since $\FF$ is countable and the desired equalities hold $\nu$-a.e., $\nu(X\setminus X_2) =0$.

Observe now that for every  $f\in C_c(X)$ and $x\in X_2$
$$Lf(x)=\lim_{n\to \infty }\frac{S_n Pf(x)}{S_n\varPhi(x)}
= \lim_{n\to \infty }\frac{S_n f(x) +P^{n+1}f(x)-f(x)}{S_n\varPhi(x)}=\lim_{n\to \infty }\frac{S_n f(x) }{S_n\varPhi(x)}=L(Pf)(x),$$
since $f$ and $Pf$ are bounded and $S_n\varPhi\to\infty$. Thus, if $x\in X_2$ and $f\in \FF$, we have
$$\nu_x(f)=Lf(x)=LPf(x)=\nu_x(Pf).$$

Fix $f\in C_c(X)$ and $\varepsilon>0$. Let $h\in \FF$ be
such that $\|f-h\|_{\infty}\leq\varepsilon\varPhi$. Thus  $P|f-h|\leq\varepsilon P\varPhi$.
Then it follows that
\begin{align*}
|\nu_x(Pf)-\nu_x(f)|&\leq |\nu_x(Pf)-\nu_x(Ph)|+|\nu_x(Ph)-\nu_x(h)|+|\nu_x(f)-\nu_x(h)|\\
&=\varepsilon\nu_x(P\varPhi)+0+\varepsilon\nu_x(\varPhi)\\
&=\varepsilon(L(P\varPhi)(x)+L\varPhi(x))\\
&=2\varepsilon \quad\mbox{ since $L(P\varPhi)(x)=L\varPhi(x)=1$}.
\end{align*}
Letting $\varepsilon\to 0$ we obtain that
 $\nu_x(Pf)=\nu_x(f)$ for all $f\in C_c(X)$. Thus $\nu_x$ is $P$--invariant.

\textbf{Step IV.}
{\it We are going to prove that there exists a set of full measure $X_3\subset X_2$ such that $\nu_x$ is ergodic for all $x\in X_3.$}

Take $f\in C_c(X)$ and  observe that $Lf$ is bounded (since $f\leq C\cdot \varPhi$ for some constant $C$, thus $Lf\leq C L\varPhi=C$) and, by the Chacon-Ornstein Theorem, invariant. By  \cite[Prop 2.5.6]{Garsia}, $P(gLf)=(Pg)\ (Lf)$ for any $g\in L^1(\nu)$ and thus $L(gLf)= (Lg)\ (Lf)$. In particular  for $\nu$-almost every $x$
$$
\nu_x(f  Lg)\stackrel{(\ref{eq: conv nux mes})}{=}L(f Lg)(x)=Lf(x) Lg(x)\stackrel{(\ref{eq: conv nux mes})}{=}\nu_x(f)\nu_x(g).
$$
Let $X_3\subseteq X_2$ be the set of all $x$ such the latter equality holds for all $f,g\in \FF$. Since $\FF$ is countable $\nu(X\setminus X_3)=0$. Take $x\in X_3$  and fix $g\in\FF$. Then $$\nu_x(f  Lg) =\nu_x(f)\nu_x(g) \quad\text{for all}\,\,f\in \FF.$$
Since $\FF$ is dense in $L^1(\nu)$, it follows that $Lg(y)=\nu_x(g)$ for $\nu_x$-almost all $y\in X$. Thus $\nu_x$ is an invariant measure such that $Lg$ is $\nu_x$-a.e. constant  and $\nu_x$ is then ergodic, by Lemma \ref{lem: eerg-const lim} applied to $\nu=\nu_x$.
The proof is completed.
\end{proof}

\noindent {\bf Acknowledgement}.
D.B. was partially supported by the National Science Center, Poland (Sonata Bis, grant number DEC-2014/14/E/ST1/00588).

\end{document}